\PassOptionsToPackage{table,dvipsnames}{xcolor}
\documentclass[10pt,journal,twocolumn,web]{article}
\usepackage{generic}
\usepackage{cite}
\usepackage{amsmath,amssymb,amsfonts,amsthm}
\usepackage{bm,mathtools,enumerate}
\usepackage{algorithmic}
\usepackage{graphicx}
\usepackage{algorithm,algorithmic}
\usepackage{hyperref}
\hypersetup{hidelinks=true}
\usepackage{textcomp}
\usepackage{csquotes}
\usepackage{bbm}
\usepackage{bm}
\usepackage{tikz}
\usetikzlibrary{arrows,calc, fit, backgrounds,decorations.pathreplacing, shapes.arrows, arrows.meta}
\usepackage{pgfplots}
\usepackage{pgfplotstable}
\usepackage{booktabs}
\usepackage{units}
\definecolor{blueUni}{RGB}{0,0,200} 
\definecolor{orangeUni}{RGB}{213,94,0}
\definecolor{greenUni}{RGB}{0,158,115}
\definecolor{redUni}{RGB}{204,121,167}

\newenvironment{alglist}
{\begin{list}{--}{
        \setlength{\leftmargin}{1em}
        \setlength{\labelwidth}{0.6em}
        \setlength{\labelsep}{0.3em}
        \setlength{\itemsep}{0pt}
        \setlength{\topsep}{2pt}
        \setlength{\parsep}{0pt}
        \setlength{\partopsep}{0pt}
}}
{\end{list}}

\newcommand{\N}{\mathbb{N}}
\newcommand{\R}{\mathbb{R}}

\newcommand{\Z}{\mathbb{Z}}

\renewcommand{\mathbf}{\bm}

\newcommand{\p}{\mathbf{p}}
\newcommand{\q}{\mathbf{q}}
\newcommand{\bu}{\mathbf{u}}
\newcommand{\w}{\mathbf{w}}
\newcommand{\x}{\mathbf{x}}

\newcommand{\z}{\mathbf{z}}
\newcommand{\dotz}{\dot{\mathbf{z}}}

\newtheorem{ass}{Assumption}
\newtheorem{cor}{Corollary}
\newtheorem{rem}{Remark}
\newtheorem{lem}{Lemma}

\newtheorem{theorem}{Theorem}
\newtheorem{definition}{Definition}
\newtheorem{proposition}{Proposition}
\newtheorem{example}{Example}
\DeclareMathOperator{\diag}{diag}

\title{%
    Model Predictive Control: Systematic Design Tailored to Nonholonomic Systems\footnote{This work was supported by the Deutsche Forschungsgemeinschaft (DFG, German Research Foundation) under Grant 545246093.}
}
\author{Emilia F. P. Krendelsberger$^{1}$\thanks{Corresponding author, e-mail: emilia.krendelsberger@tu-ilmenau.de}, Jasmin Krauspenhaar$^{2}$,  {L}ea Bold$^{1}$, Mario Rosenfelder$^{2}$,\\
Peter Eberhard$^{2}$, Henrik Ebel$^{3}$, and Karl Worthmann$^{1}$\\
        \small $^{1}$Optimization-based Control Group, Institute of Mathematics, Technische Universität Ilmenau, Germany \\
        \small $^{2}$Institute of Engineering and Computational Mechanics~(ITM), University of Stuttgart, Germany \\
        \small $^{3}$Department of Mechanical Engineering, LUT University, Lappeenranta, Finland\\
}
\date{}
\begin{document}

\maketitle

%

\begin{abstract}
Nonholonomic systems are ubiquitous in applications such as transportation and robotics, yet their control remains challenging.  In order to achieve asymptotic stability of a desired equilibrium, model predictive control (MPC) schemes without stabilizing terminal conditions have recently been proposed for nonholonomic kinematic systems. However, this approach still required lengthy system-specific calculations. This paper eliminates the need for these auxiliary calculations, as we propose a streamlined systematic design for MPC stage costs tailored to nonholonomic systems, utilizing the concepts of privileged coordinates and nonholonomic weights. 
Furthermore, the proposed design methodology is extended beyond purely kinematic models to a wider class of systems with non-negligible dynamics, and we highlight the geometric meaning of the resulting MPC stage costs. The systematic procedure is demonstrated using an articulated vehicle as an example.
\end{abstract}

\begin{keywords}
{Model}Predictive Control, Nonholonomic System, Robotics, Homogeneous Approximation
\end{keywords}

\maketitle

\section{Introduction}
\label{sec:introduction}
Nonholonomic systems arise naturally in many engineering applications and are characterized by non-integrable constraints on their velocities, see, e.g.,~\cite{QinEtAl22,LaumondSekhavatLamiraux05}.
Although such systems are ubiquitous in practice, see, e.g.,~\cite{DeLucaOrioloSamson05,Bloch15}, their control remains challenging. 
In general, linearization does not preserve controllability and violation of Brockett's necessary condition~\cite{Brockett83} rules out the existence of Lipschitz-continuous, time-invariant asymptotically stabilizing state feedback laws for nonholonomic systems~\cite{Coron90,Sontag90}. 
These fundamental challenges have motivated the development of sophisticated setpoint control strategies over the past decades~\cite{DeLucaOrioloSamson05,jean_control_2014}, proposing either time-varying~\cite{MurrSast93,Zuyev16,GrushkovskayaZuyev24,MeraRios24} or discontinuous but time-invariant feedback laws~\cite{Astolfi96,Zeng22}. 
These works include also recent contributions, e.g., to obtain uniform exponential convergence for the prototypical Brockett integrator with a discontinuous state feedback~\cite{Zeng22}, or to obtain bounded controls for finite-time regulation of one nonholonomic system~\cite{MeraRios24}. 

Such methods are commonly tailored to specific nonholonomic systems, and typically require substantial problem-dependent tuning effort. 
Consequently, a general and systematic controller design methodology for nonholonomic systems remains elusive. 

This motivates the use of optimization-based control schemes such as model predictive control~(MPC), where the system dynamics and actuation constraints can be incorporated directly into the controller design. 
Deriving provably stabilizing MPC schemes for nonholonomic systems is, however, far from straightforward.
In MPC, asymptotic stability is often established through stabilizing terminal conditions, 
which are commonly constructed based on the system's linearization~\cite{ChenAllgoewer98}, an approach that does, in general, not work for nonholonomic systems as shown in~\cite{muller2017quadratic}. 
Hence, the construction of terminal conditions is significantly more intricate 
for 
nonholonomic systems.
Whereas stabilizing terminal conditions 
have been developed for nonholonomic systems, see, e.g., \cite{GuHu06,XieFierro08,KuhneLagesDaSilva05}, they are 
tailored to specific systems and usually rely on hand-designed terminal controllers and associated terminal sets. 
Proposed terminal sets are then usually not only hand-crafted per system but also, compared to terminal sets for holonomic systems, quite unusual in the sense that they are commonly nonconvex and may not even contain the setpoint in the interior, encompassing instead states from which specific steering commands are known that steer the system to the setpoint. 
These difficulties motivate the study of MPC schemes without stabilizing terminal ingredients, see, e.g., \cite{grune_nonlinear_2017}, simplifying and robustifying numerical solution as no (additional) state constraints are introduced. 
However, for nonholonomic systems, conventional quadratic stage costs
%
generally do not guarantee asymptotic stability of the resulting MPC closed loop, see~\cite{muller_quadratic_2017,RoseEbel22}. 
For the differential-drive robot, arguably the most prominent wheeled mobile robot, non-quadratic stage costs have been constructed explicitly to establish suitable growth bounds on the value function and, as a consequence, a relaxed Lyapunov inequality to show asymptotic stability of the MPC closed loop~\cite{WorthmannEtAl16,WorthmannEtAl15}.
The works~\cite{RosenfelderEbelEberhard21,RosenfelderEbelEberhard22a} provide a mechanical interpretation of the resulting mixed-exponents cost. 
Control directions directly accessible through the system inputs are penalized less strongly close to the desired steady state than directions that can only be generated 
through Lie brackets.
This insight has subsequently led to the development of tailored mixed-exponents stage costs for MPC-based setpoint control of nonholonomic systems without terminal ingredients~\cite{RoseEbel22,ebel_note_2023}, building upon results for homogeneous control systems~\cite{coron_model_2019}. 
However, the overall procedure requires the construction of a homogeneous approximation of the nonholonomic kinematics and several auxiliary quantities, which is rather 
technical and lengthy.

The present paper eliminates the need for these lengthy intermediate constructions and, as its main contribution, provides a compact and systematic design procedure for model predictive controllers without terminal ingredients for nonholonomic systems.
The resulting MPC scheme ensures local asymptotical stability of the closed-loop system.
To this end, tools from sub-Riemannian geometry are employed to express the system in privileged coordinates, cf.~\cite{bellaiche1996tangent,jean_control_2014}. 
Based on this representation, it is rigorously shown that a homogeneous approximation can always be obtained for the system class with the associated homogeneity parameters coinciding with the so-called nonholonomic weights.
These weights constitute intrinsic geometric characteristics of the nonholonomic system and can be determined in a straightforward manner.
The framework yields tailored mixed-exponents stage costs for general nonholonomic systems that can be constructed solely from a set of privileged coordinates and the corresponding weights.
As a further novelty, we not only consider nonholonomic kinematics described by driftless control-affine systems but also treat kinematic systems with a so-called dynamic extension.
In particular, we extend the ideas of~\cite{ebel_note_2023} from constant to configuration-dependent dynamic extensions and generalize the corresponding MPC design procedure, including stability guarantees, to this broader system~class.

\definecolor{ColITMorange}{RGB}{252, 109, 35}
\definecolor{ColITMViolet}{RGB}{120, 120, 255}%
\definecolor{UniBlueDark}{RGB}{0,81,158}
\definecolor{UniBlueLight}{RGB}{0,190,255}
\definecolor{UniAnthrazit}{RGB}{62, 68, 76}
\colorlet{questionCol}{ColITMViolet}
\colorlet{mainCol}{ColITMorange}

\def\Distx{0.4cm}
\def\Disty{0.15cm}
\def\DistyText{0.9cm}
\def\DistyProblemFormulation{0.6cm}
\def\shiftxArrowProp{1.2cm}
\def\shiftyArrowPropOne{1.15mm}
\def\shiftyArrowPropTwo{6.1mm}
\def\DistArrow{-1.95cm}

    \tikzset{
        TopicProblemFormulation/.style = {fill=white, minimum height=0.6cm, fill opacity=0.2, line width = 0.1mm, draw opacity=1, text opacity=1, draw=black, align=center, text width=8.5cm,  font=\small},
    }

    \tikzset{
        TopicLEFT/.style = {fill=white, fill opacity=1, line width = 0.1mm, draw opacity=1, text opacity=1, draw=black, align=center, text width=3.2cm,  font=\small},
    }

    \tikzset{
        TopicRIGHT/.style = {fill=white, fill opacity=0.2, line width = 0.1mm, draw opacity=1, text opacity=1, draw=black, align=center, text width=4.6cm,  font=\small},
    }

    \tikzset{
        TopicBackground/.style = {fill=cyan, draw=cyan, fill opacity=0.1, rounded corners=2pt, inner sep=2pt},
    }

    \tikzset{
        TopicImage/.style = {line width = 0.1mm, draw opacity=1, draw=white, inner sep = 0pt, outer sep = 0pt, anchor=west},
    }

	\tikzset{
        ArrowsA/.style = {->, line width = 0.2mm},
    }
    \tikzset{
        ArrowsKin/.style = {dashed, line width = 0.2mm, draw = cyan},
    }

\begin{figure}
\centering
    \resizebox{\columnwidth}{!}{%
    \input{figures/graphicalAbstract.tex}
        }
    \caption{Graphical abstract of the main contributions of this paper.}
    \label{fig:overview}
\end{figure}

Figure~\ref{fig:overview} provides a graphical overview of the paper's structure and illustrates the relationships between its main components.
Following this overview, Section~\ref{sec:Problem Formulation} introduces the considered nonholonomic systems 
and establishes the control setup for local continuous-time MPC. 
Then, Section~\ref{sec:Stabilizing Model Predictive Control} presents the systematic controller design procedure and two novel theorems establishing local asymptotic stability of the setpoint under the MPC closed loop using mixed-exponents stage costs based on the nonholonomic weights of the system. 
The proof of stability is provided, drawing on propositions which state that for both considered system classes, there exists a homogeneous approximation. 
The construction of these approximations for both system classes, alongside the proof of the aforementioned propositions, is given in Sections~\ref{sec:Homogeneous Approximation for Kinematic Systems} and \ref{sec:Homogeneous Approximation for Kinematic Systems with Dynamic Extension}, respectively. 
As this article draws from theory from different fields, with different notational conventions, in the interest of self-containedness and accessibility to the reader, we (re)introduce some concepts from sub-Riemannian geometry in Section~\ref{sec:Stabilizing Model Predictive Control}, with notation consistent to the rest of the article.
The example of an articulated vehicle is presented in Section~\ref{sec:example}, illustrating the control-design steps and the effectiveness of the proposed controller in simulation.
We summarize and evaluate the results in Section~\ref{sec:Conclusion}.
\\

\noindent \textit{Notation}: We write $\mathbb{Z}_{a:b}$ with $a,b\in\Z$ for the set $\{k\in\Z:a\leq k\leq b\}$. Vectors and vector-valued functions are generally written in bold type.  
We denote the exponents of monomials with a multi-index $\bm{\nu}$, i.e., $\z^{\bm{\nu}}=\Pi_{i=1}^{n}z_{i}^{\nu_i}$. 
Further, $\bm{0}_n$ and $\bm{1}_{n}$ denote a vector of length $n$ containing only zeros or ones, respectively, and $\bm{0}_{n\times m}$ is the zero matrix of dimension $n\times m$. 
The identity matrix of dimension $n\times n$ is $\bm{I}_n$. We employ the following symbols for spaces of continuous functions, $\mathcal{K}$ for functions $\alpha:\R_{\geq 0}\rightarrow\R_{\geq 0}$, where $\alpha$ is strictly increasing with $\alpha(0)=0$, $\mathcal{K}_{\infty}$ for functions $\alpha\in\mathcal{K}$ where $\lim_{s\rightarrow\infty}\alpha(s)=\infty$, $\mathcal{L}$ for functions $\delta:\R_{\geq 0}\rightarrow\R_{\geq 0}$, where $\delta$ is strictly decreasing with $\lim_{t\rightarrow\infty}\delta(t)=0$, and $\mathcal{KL}$ for functions $\beta:\R_{\geq 0}\times\R_{\geq 0}\rightarrow\R_{\geq 0}$, where $\beta(\cdot,t)\in\mathcal{K}$ for all $t\geq 0,$ and $ \beta(r,\cdot)\in\mathcal{L}$ for all $r\geq 0$.

\section{Problem Formulation}
\label{sec:Problem Formulation}

Our aim in this paper is the stabilization of a desired setpoint for nonholonomic systems via model predictive control (MPC).
In this section, we first introduce 
two specific classes of nonholonomic systems, 
namely driftless kinematic systems and their counterparts with dynamic extension.
Subsequently, we present an MPC setup that is amenable to 
both system classes and provide the problem formulation.

\begin{table}[h] 
\setlength{\tabcolsep}{3pt}
\begin{tabular}{p{30pt}p{90pt}p{105pt}}
        \toprule
        Symbol & Name & Description\\
        \midrule $T_{\q}\R^{n_q}$& tangent space& space of all possible velocities for a system with position $\q\in\R^{n_q}$, cf.~\cite{Bloch15}\\
        \midrule $\mathcal{L}^{b}_{\bm{g}}f$& Lie derivative/\newline nonholonomic derivative&  see Section~\ref{subsec:Driftless Nonholonomic Systems}\\
        \midrule
        $[\bm{f},\bm{g}]$& Lie bracket& see Section~\ref{subsec:Driftless Nonholonomic Systems}\\
        \midrule 
        $\Delta^k(\p)$
        & 
        distributions & see Section~\ref{subsec:Driftless Nonholonomic Systems}\\
        \midrule
        $r(\p)$ 
        & 
        degree of \newline nonholonomity & see Section~\ref{subsec:Driftless Nonholonomic Systems}\\
        \midrule
        $\w(\p)$
        & 
        nonholonomic weights & see Section~\ref{subsec:Systematic Design Procedure for Kinematic Systems}\\
        \midrule
         $\z(\q;\p)$ 
         & 
         privileged \newline coordinates & see Section~\ref{subsec:Systematic Design Procedure for Kinematic Systems}\\
        \midrule
        $\omega(\bm{\nu};\p)$ 
        & 
        weighted degree &  see Definition~\ref{def:weighted_degree} \\
        \midrule
        $\operatorname{ord}_{\p}(f)$
        & 
        nonholonomic order \newline of a function & see Section~\ref{subsec:Systematic Design Procedure for Kinematic Systems} \\
        \midrule
        $\operatorname{ord}_{\p}(\bm{g})$
        & 
        nonholonomic order \newline of a vector field & see Definition~\ref{def:nonhol_ord_vec_fields_and_operators}\\
        \bottomrule
    \end{tabular}
    \caption{Important characteristics for nonholonomic systems} 
    \label{tab:nonholsys}
\end{table}

\subsection{Driftless Nonholonomic Kinematic Systems}
\label{subsec:Driftless Nonholonomic Systems}

We consider the driftless nonholonomic kinematic system
\begin{align}
	\dot{\bm{q}}(t) = \sum_{j=1}^{n_v} \bm{g}_j(\bm{q}(t)) v_j(t) \eqqcolon \bm{G}(\bm{q}(t)) \bm{v}(t), \label{eq:kinematics_nonholonomic_driftless}
\end{align}
where the state~$\bm{q}(t) \in \mathbb{R}^{n_q}$ and the input~$\bm{v}(t)~\in~\mathbb{V}~\subset \mathbb{R}^{n_v}$ denote the position and the velocity of the system at time~$t$, respectively. 
We assume that $n_v < n_q$ holds, meaning that 
the system is underactuated. 
For notational convenience, we suppress the time argument~$t$ in the following.

The control vector fields $\bm{g}_j~:~\mathbb{R}^{n_q}~\to~\mathbb{R}^{n_q}$, $j~\in~\Z_{1:n_v}$, represent the columns of the matrix-valued map~$\bm{G}$. 
For the following definitions, we assume that the control vector fields~$\bm{g}_j$ are sufficiently smooth, before we make our smoothness requirements more precise in Assumption~\ref{ass:kin_sys_setpoint}. 
The distribution~$\Delta^1$ 
associated with the control vector fields~$\bm{g}_j$ is defined 
by $
\bm{q} \mapsto \operatorname{span} \{\bm{g}_1(\bm{q}), \ldots, \bm{g}_{n_v}(\bm{q})\} \subset T_{\bm{q}}\mathbb{R}^{n_q}$ (cf. Table~\ref{tab:nonholsys}) for position~$\q \in \R^{n_q}$.
This captures the set of instantaneously admissible velocities of the system at~$\bm{q}$.
Motion directions not contained in~$\Delta^{1}(\bm{q})$ have to be generated by Lie brackets defined as 
\begin{align*}
	[\bm{g}_i, \, \bm{g}_j] (\bm{q}) \coloneqq \left(\frac{\partial \bm{g}_j}{\partial \bm{q}}\ \bm{g}_i - \frac{\partial \bm{g}_i}{\partial \bm{q}}\ \bm{g}_j\right) (\bm{q}).
\end{align*}
Higher-order Lie brackets correspond to increasingly intricate maneuvers required to steer the system.
We recursively define the distributions
\[\Delta^i\coloneqq\operatorname{span}\{\bm{g}_j,\left[\bm{g}_k, \bm{g}_j \right] \, \vert \, \bm{g}_k\in\Delta^1,\ \bm{g}_j\in\Delta^{i-1} \}\]
for $i\geq 2$. 
The integer
\begin{align*}
  r&=r(\bm{q})\\
  &\coloneqq\min\{i\in\N\,|\,\forall\, \bm{g}_k,\bm{g}_j \in \Delta^i(\bm{q}) : \left[ \bm{g}_k,\bm{g}_j \right] \in \Delta^{i}(\bm{q})\}  
\end{align*} is called the degree of nonholonomy, which serves as an indicator of the system's complexity; $\Delta^{r}(\p)$ is subsequently referred to as the accessibility distribution. 
Finally, the driftless system~\eqref{eq:kinematics_nonholonomic_driftless} is locally controllable at a point~$\p\in\R^{n_q}$, which is an equilibrium of the driftless system, if in a neighborhood of~$\p$ the accessibility distribution spans the state space, i.e., $\text{dim}(\Delta^r(\bm{q}))=n_q$, which is equivalent to the Lie Algebra Rank Condition (LARC), cf.~\cite{coron_control_2009}. If this conditions holds globally, the nonholonomic system~\eqref{eq:kinematics_nonholonomic_driftless} is (globally) controllable  --~a property which is not preserved by the Jacobian linearization, as linear combinations of the control vector fields are insufficient to span the entire state space, since the system is driftless and underactuated.
\subsection{Nonholonomic Kinematic Systems with Dynamic Extension}
\label{subsec:Dynamic Extension of Nonholonomic Kinematics}

Consider kinematic systems~\eqref{eq:kinematics_nonholonomic_driftless} with a dynamic extension 
\begin{align}\label{eq:dynamic_part_dynamic_system}
 \bm{M}(\bm{q}) \dot{\bm{v}} = \bm{B}(\bm{q})\bm{u}
\end{align}
with input~$\bm{u}\in\mathbb{U}\subset\mathbb{R}^{n_u}$ and  matrix~$\bm{B}(\bm{q})\in\mathbb{R}^{n_v \times n_u}$. Moreover, $\bm{M}$ is a symmetric positive-definite matrix-valued function $\bm{M}\!: \mathbb{R}^{n_q} \to \mathbb{R}^{n_v \times n_v}$, i.e., there exists a $\mathcal{K}_{\infty}$-function $\alpha$ such that, for all $\bm{q}\in\R^{n_q}$ and $\bm{v}\in\R^{n_v}$, $\bm{v}^{\top}\bm{M}(\bm{q})\bm{v}\geq \alpha(\lVert \bm{v}\rVert)$ holds.\footnote{Strictly speaking, it will later only be required that~$\bm{M}$ is locally of full rank, but major applications in mechanics give positive definiteness directly. Hence, we focus on that case in the following.}
The overall system then reads
\begin{align}\label{eq:dynamic_extension_system}
     \begin{bmatrix}
        \dot{\bm{q}} \\ \bm{M}(\bm{q}) \dot{\bm{v}}
    \end{bmatrix} 
    = 
    \begin{bmatrix}
        \bm{G} (\bm{q}) \\ \bm{0}_{n_v\times n_v} 
    \end{bmatrix}\bm{v} 
    +
    \begin{bmatrix}
        \bm{0}_{n_q\times n_u} \\ \bm{B}(\bm{q})
    \end{bmatrix}
    \bm{u}
\end{align}
with 
state~$\bm{x}=[\bm{q}^\top \ \bm{v}^\top
]^\top \in \mathbb{R}^{n_x}$, where $n_x=n_q + n_v$. 

For kinematic systems with and without dynamic extension, we impose the following assumptions throughout this paper.
\begin{ass}
\label{ass:kin_sys_setpoint}
    For the underactuated kinematic system~\eqref{eq:kinematics_nonholonomic_driftless} and a setpoint $\p\in\R^{n_q}$, let the following hold:
    \begin{enumerate}[(i)]
        \item The control vector fields $\bm{g}_j$, $j \in \mathbb{Z}_{1:n_v}$, are non-zero at~$\p$ and analytic 
        in a neighborhood~$\mathcal{N}(\p)$ of ~$\p$.\footnote{Note that local analyticity implies that $g_j$ is locally $\mathcal{C}^\infty$ and also implies the existence of $R>0$ such that the Taylor series of $\bm{g}_j$ converges absolutely to $\bm{g}_j(\bm{q})$ for all $\bm{q}$ with \\$\lVert \bm{q}-\p\rVert <R$.}
        \item $\text{dim}(\Delta^r(\bm{p}))=n_q$ (LARC)
        \item The setpoint $\p$ is regular, i.e., the dimensions of the distributions~$\Delta^i$, $i\in\mathbb{Z}_{1:r(\p)}$, are constant on~$\mathcal{N}(\p)$. 
    \end{enumerate}
    If the kinematic system\footnote{Note that, then, the setpoint consists of position \textit{and} velocity such that also the neighbourhood is implicitly adapted.} has a dynamic extension~\eqref{eq:dynamic_part_dynamic_system}, let, in addition, the following hold on~$\mathcal{N}(\p)$: 
    \begin{enumerate}[(i)]
    \setcounter{enumi}{3}
        \item The dynamic extension~\eqref{eq:dynamic_part_dynamic_system} is fully actuated, i.e., $n_u = n_v$ and $\bm{B}(\bm{q})$ is regular for all $\bm{q} \in \mathcal{N}(\bm{p})$. 
        \item The map $\bm{q} \mapsto \bm{M}^{-1}(\bm{q})\bm{B}(\bm{q})$ is analytic. 
    \end{enumerate}
    \end{ass}
If Assumption~\ref{ass:kin_sys_setpoint} holds, the system with dynamic extension~\eqref{eq:dynamic_extension_system} is (locally) controllable, 
see, e.g.,~\cite{Sussmann91, Sicilano09}.

\subsection{Model Predictive Control}
\label{sec:Model Predictive Control Setup}

Next, we consider MPC without stabilizing terminal conditions for the 
system classes presented in Sections~\ref{subsec:Driftless Nonholonomic Systems} and~\ref{subsec:Dynamic Extension of Nonholonomic Kinematics}.
In the following, for ease of notation, we employ the 
nonlinear control system $\dot{\bm{x}} = \bm{f}(\bm{x}(t), \bm{u}(t))$ with state $\bm{x}(t) \in \mathbb{R}^{n_x}$ to cover both nonholonomic kinematic systems with ($n_x=n_q+n_v$) and without ($n_x=n_q$) dynamic extension in a unified manner. 
The set~$\mathbb{U} \subset \R^{n_u}$ of feasible control values is supposed to contain the origin in its interior, while being compact and convex. 
Further, for an initial value $\hat{\x} \in \mathcal{N}(\p)$, the input function~$\bu: [0,T] \rightarrow \R^{n_u}$ is said to be admissible, denoted by $\bu \in \mathcal{U}_T(\hat{\x})$, if $\bm{u}$ is measurable, satisfies the control constraints $\bm{u}(t) \in \mathbb{U}$, $t \in [0,T]$, and yields existence of a unique solution $\x(\cdot;\hat{\x},\bm{u})$ on $[0,T]$.

At the core of the MPC closed loop will be the (continuous-time) optimal control problem~(OCP) parametric in the optimization horizon~$T$, $T>0$, and the initial state $\bm{x} (t) \in \mathcal{N}(\p)$ (at time~$t \geq 0$), where $\p$ stands for the setpoint of Assumption~\ref{ass:kin_sys_setpoint}:
\begin{subequations}\label{eq:MPC_cont}%
    \begin{alignat}{3}%
            & \underset{\bm{u}(t+\cdot \mid t) \in \mathcal{U}_T(\bm{x}(t\,\vert \, t))}{\textnormal{minimize}}
            &&\ \int_{t}^{t+T} \ell\!\left( \bm{x}(\tau\,\vert \, t), \bm{u}(\tau\,\vert \, t) \right)\, \textnormal{d}\tau \label{eq:MPC_cost_cont}\\
            & \quad\:\:\:\text{subject to}
            &&\ \dot{\bm{x}}(\tau \,\vert \, t) = \bm{f}( \bm{x}(\tau\,\vert \, t),  \bm{u}(\tau\,\vert \, t)), \label{eq:MPC_dynamics_cont}\\
            &&&\ \bm{x}(t\,\vert \, t) = \bm{x} (t).\label{eq:MPC_initial_condition_cont}
    \end{alignat}%
\end{subequations}%
The conventional MPC notation is used, where $\bm{x} (\tau\,\vert \, t)$ denotes the state at time $\tau$, $\tau \in [t,t+T]$ predicted at time~$t$ for the admissible input function $\bm{u}(\cdot \mid t)$.
The stage costs~$\ell:\R^{n_x}\times\R^{n_u}\rightarrow \R_{\geq 0}$ are continuous and there exist $\mathcal{K}_{\infty}$-functions $\alpha_1,\alpha_2$ such that
\[
    \alpha_1(\lVert \x\rVert) \leq l^{\star}(\x) := \min_{\bm{u}\in\mathbb{U}} \ell(\x,\bu) \leq \alpha_2(\lVert \x\rVert) \quad\forall\,\x\in\R^{n_x}.
\]
Next, we show that a minimizer of OCP~\eqref{eq:MPC_cont} exist, which will later correspond to initial feasibility in MPC. 
While we show the assertion only for the driftless nonholonomic kinematic system~\eqref{eq:dynamic_part_dynamic_system} of Subsection~\ref{subsec:Driftless Nonholonomic Systems} provided that Assumption~\ref{ass:kin_sys_setpoint} holds, the proof can be extended to systems with dynamic extensions leveraging the controllability imposed in Assumption~\ref{ass:kin_sys_setpoint} if the stage costs and the initial value are chosen such that existence of an admissible control function with finite cost is ensured.\footnote{This holds for the representative mechanical systems considered in the following sections.} 
We only sketch the proof following standard arguments in optimal control, see, e.g., \cite[Theorem 4.6]{Berger22} and~\cite[Proposition 2]{Sakamoto23} for additional details.
\begin{proposition}[Initial feasibility]
    Let the set~$\mathbb{U}$ of feasible control values be convex, compact and contain the origin in its interior.
    Further, let Assumption~\ref{ass:kin_sys_setpoint} hold for system~\eqref{eq:MPC_dynamics_cont} given by the driftless nonholonomic kinematic system~\eqref{eq:kinematics_nonholonomic_driftless}. 
    Then, for each initial state $\hat{\x} \in \mathcal{N}(\p)$ and optimization horizon~$T > 0$, 
    the infimum of the OCP~\eqref{eq:MPC_cont} is attained, i.e., 
    an optimal solution~$\bm{u}^\star = \bm{u}_T^\star(\hat{\x})$ exists.
\end{proposition}
%
\begin{proof}
W.l.o.g.\ we set $t=0$ and define the cost functional
\begin{align*}
   J_{T}(\hat{\x},\bu)&\coloneqq \int_{0}^{T} \ell\!\left( \bm{x}(\tau\,\vert \, 0), \bm{u}(\tau\,\vert \, 0) \right)\, \textnormal{d}\tau.
\end{align*}
If the solution~$\bm{x}(\cdot\,\vert \, 0)$ does not exist on $[0,T]$ or the 
stage costs are not integrable, we set $J_T(\hat{\x},\bu)=\infty$. 
Since $\bm{u}(t) \equiv 0$ is an admissible control function (for $\hat{\bm{x}} \in \mathcal{N}(\hat{\bm{x}})$),
the set~$\mathcal{U}_T(\hat{\bm{x}})$ of admissible control functions for which additionally $J_T(\hat{\x},\bu)<\infty$ holds, 
is nonempty.
Further, we have $J_T(\hat{\bm{x}},\bm{u}) < \infty$, while the cost functional is bounded from below due to positive definiteness of the stage cost~$\ell$. Hence, 
\[
    J_{T}^{\star}(\hat{\x})\coloneqq\inf_{\bu\in\mathcal{U}_{T}}J_{T}(\hat{\x},\bu) \in [0,\infty)
\] 
holds for the cost functional~$J_T^\star$ such that an infimizing sequence $(\bu_k)_{k \in \N_0} \subset \mathcal{U}_T(\hat{\bm{x}})$ exist. 
Then, using standard arguments, see, e.g., \cite[Theorem 4.6]{Berger22}, invoking the imposed assumptions on the feasible set~$\mathbb{U}$, it can be shown that there exists a control $\bu^\star\in\mathcal{U}_{T}(\hat{\bm{x}})$ such that $J_{T}(\hat{\x},\bu^{\star}) = \min_{\bu \in \mathcal{U}_T(\hat{\bm{u}})} J_{T}(\hat{\x},\bu)$.
\end{proof}

For sampling time~$\delta_t$, $\delta_t \in (0,T)$, we define the MPC feedback law $\mu_{T,\delta_t}: [0,\delta_t) \times \mathcal{N}(\p)\to \mathbb{U}$ by
\[
    \mu_{T,\delta_t}(\tau;\hat{\bm{x}}) := \bm{u}^\star(\tau) = \bm{u}^\star_T(\tau;\hat{\bm{x}}) \qquad\forall\,\tau \in [0,\delta_t),
\]
where the dependency of the optimal control~$\bm{u}^\star$ on the initial state~$\hat{\bm{x}}$ and the optimization horizon~$T$ is shown, while the time~$t$ is omitted since the OCP~\eqref{eq:MPC_cont} does not depend on it (w.l.o.g.\ $t = 0$).
Using the feedback law~$\mu_{T,\delta_t}$ results in the MPC closed loop governed by
\begin{align}\label{eq:MPC_CL}    
    &\dot{\bm{x}}_{\mu_T,\delta_t}(t) \nonumber\\
    &= f(\bm{x}_{\mu_T,\delta_t}(t), \mu_{T,\delta_t}(t - \delta_t \lfloor t/\delta_t \rfloor; \bm{x}_{\mu_T,\delta_t}(\delta_t \lfloor t/\delta_t \rfloor)) 
\end{align}
with $\bm{x}_{\mu_T,\delta_t}(0) = \bm{x}(0)$,
assuming solvability of the OCP~\eqref{eq:MPC_cont} at each sampling time~$t_i = i \delta_t$, $i \in \N_0$.

For the driftless kinematic system~\eqref{eq:kinematics_nonholonomic_driftless}, each state $\bm{q}\in\R^{n_q}$ is an equilibrium ($\bm{u} = 0$). 
If there is a dynamic extension, this extends to states with a velocity of zero, where each state $[\bm{q}^\top\,\bm{0}_{n_v}^{\top}]^{\top},\ \bm{q}\in\R^{n_q},$ is an equilibrium of the system~\eqref{eq:dynamic_extension_system}.
Our goal is to stabilize a desired setpoint~$\p$ using MPC, i.e., we want to ensure that $\p$ is asymptotically stable w.r.t.\ the MPC closed loop. 
To make that statement precise, we conclude this section by defining asymptotic stability of the closed loop, cf.~\cite{khalil_nonlinear_2002}. 
\begin{definition}\label{def:asymptotic_stability}
    The equilibrium~$\bm{p}\in\R^{n_x}$ is said to be locally asymptotically stable w.r.t.\ the MPC closed loop~\eqref{eq:MPC_CL} if there exist $\beta \in \mathcal{KL}$ and $\delta >0$ such that, 
    \begin{equation}
        \left\lVert \bm{x}_{\mu_{T,\delta_t}}(t;\hat{\bm{x}}) - \bm{p} \right\rVert \leq \beta ( \left\lVert \hat{\bm{x}} - \bm{p} \right\rVert,t) \quad  \forall\, 
        \hat{\bm{x}} \in B_{\delta}(\bm{p}),
    \end{equation}
    for all $t \geq 0$ and the MPC closed-loop trajectory $\bm{x}_{\mu_{T,\delta_t}}(\cdot;\hat{\bm{x}})$, where $B_{\delta}(\bm{p})$ is the open ball with radius~$\delta$ centered at~$\bm{p}$.
\end{definition}

\section{Stabilizing Model Predictive Control}\label{sec:Stabilizing Model Predictive Control}

In~\cite{muller_quadratic_2017,RoseEbel22}, it has been shown that model predictive control with quadratic stage costs fails to ensure local asymptotic stability for driftless nonholonomic kinematic systems.
Our goal is to systematically design stage costs leveraging homogeneity such that asymptotic stability is guaranteed for the MPC closed loop 
for nonholonomic kinematic systems with and without dynamic extension.
In this section, we introduce our main results on the MPC of nonholonomic systems, namely a systematic procedure for constructing tailored stage costs solely from the underlying geometry of the system.
The proposed construction relies on two key concepts introduced below, namely nonholonomic weights and privileged coordinates.

Figure~\ref{fig:systematic_design} provides a roadmap for this design procedure.
We consider driftless nonholonomic kinematic systems, introduce the required definitions and notations, and then state the main result on local asymptotic stability of the MPC closed loop. 
\def\Distx{0.6cm}
\def\Disty{-1.8cm}
\definecolor{matchingRed}{RGB}{220,120,120}
\definecolor{UniBlueDark}{RGB}{62, 68, 76}
\definecolor{UniBlueDark}{RGB}{0,81,158}
\definecolor{blueUni}{RGB}{0,0,200} 
\definecolor{orangeUni}{RGB}{213,94,0}
\definecolor{redUni}{RGB}{204,121,167}
\tikzset{
block/.style = {rectangle, minimum width=3.5cm, minimum height=1cm, text centered, draw=black, fill=white, text width=3cm, rounded corners=2pt},
block2/.style = {rectangle, minimum width=6.5cm, minimum height=1cm, text centered, draw=black, fill=white, text width=3cm, rounded corners=2pt},
background/.style = {fill=blueUni!15, draw=blueUni, fill opacity=1, rounded corners=2pt, inner sep=2pt},
background2/.style = {fill=greenUni!50, draw=greenUni, fill opacity=0.2, rounded corners=2pt, inner sep=2pt},
arrow/.style = {->, line width = 0.2mm, thick}
}
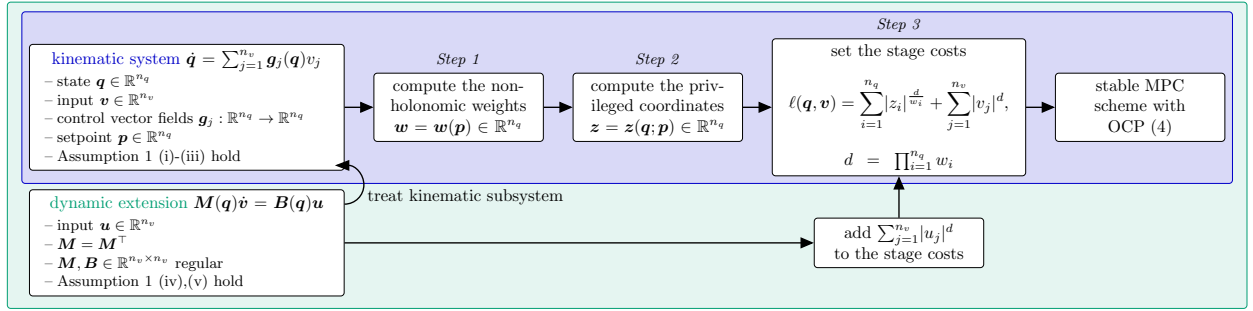
\begin{figure*}
    \centering
    \noindent\makebox[\linewidth][l]{%
        \resizebox{\linewidth}{!}{%
            \begin{tikzpicture}[>=triangle 45]
\node (system_class) [block2, text width= 6cm]
{\textcolor{blueUni}{kinematic system} $\dot{\bm{q}}=\sum_{j=1}^{n_v}\bm{g}_j(\bm{q})v_j$\\
{\small
        \begin{list}{--}{
        \setlength{\leftmargin}{1em}
        \setlength{\labelwidth}{0.6em}
        \setlength{\labelsep}{0.3em}
        \setlength{\itemsep}{0pt}
        \setlength{\topsep}{0pt}
        \setlength{\parsep}{0pt}
        }
        \item state $\bm{q}\in\R^{n_q}$
        \item input $\bm{v}\in\R^{n_v}$
        \item control vector fields $\bm{g}_j:\R^{n_q}\rightarrow\R^{n_q}$
        \item setpoint $\p\in\R^{n_q}$
        \item Assumption~\ref{ass:kin_sys_setpoint} (i)-(iii) hold
        \end{list}
    }
    };
\node (dyn_ext) [block2, text width = 6cm, below of=system_class, yshift=\Disty]
{\textcolor{greenUni}{dynamic extension}
$\bm{M}(\bm{q})\dot{\bm{v}}=\bm{B}(\bm{q})\bm{u}$\\
{\small
        \begin{list}{--}{
        \setlength{\leftmargin}{1em}
        \setlength{\labelwidth}{0.6em}
        \setlength{\labelsep}{0.3em}
        \setlength{\itemsep}{0pt}
        \setlength{\topsep}{0pt}
        \setlength{\parsep}{0pt}
        }
        \item input $\bm{u}\in\R^{n_v}$
        \item $\bm{M}=\bm{M}^\top$
        \item $\bm{M},\bm{B}\in\R^{n_v \times n_v}$ regular
        \item Assumption~\ref{ass:kin_sys_setpoint} (iv),(v) hold
        \end{list}
    }
    };
\node (weights) [block, anchor = west, xshift=\Distx] at (system_class.east)
{compute the nonholonomic weights\\
$\bm{w}=\w(\p)\in\R^{n_q}$};
\node[anchor=south] at (weights.north)
    {\small{\textit{Step 1}}};
\node (priv) [block, xshift=\Distx, anchor = west] at (weights.east)
{compute the privileged coordinates $\z=\z(\bm{q};\bm{p})\in\R^{n_q}$};
\node[anchor=south] at (priv.north)
    {\small{\textit{Step 2}}};
\node (stage_cost_1) [block, text width = 5cm, anchor = west, xshift=\Distx] at (priv.east)
{set the stage costs 
    \begin{equation*}
        \ell(\bm{q},\bm{v}) =
        \sum_{i=1}^{n_q} \lvert z_i \rvert^{\frac{d}{w_i}}+
        \sum_{j=1}^{n_v} \lvert v_j\rvert^{d},        
    \end{equation*}
    $d=\prod_{i=1}^{n_q}w_i$
    };
 \node(step3)[anchor=south] at (stage_cost_1.north)
    {\small{\textit{Step 3}}};
\node (stage_cost_2) [block, below of=stage_cost_1, yshift=\Disty]
{add $\sum_{j=1}^{n_v}\lvert u_j\rvert^{d}$ to the stage costs};
\node (MPC) [block, anchor = west, xshift=\Distx] at (stage_cost_1.east)
{stable MPC scheme with OCP~\eqref{eq:MPC_cont}};
\draw[arrow] (system_class)--(weights);
\draw[arrow] (weights)--(priv);
\draw[arrow] (priv)--(stage_cost_1);
\draw[arrow] (stage_cost_1)--(MPC);
\draw[arrow] (dyn_ext)--(stage_cost_2);
\draw[arrow] (stage_cost_2)--(stage_cost_1);
\coordinate (stage_cost_1_top) at ($(stage_cost_1.north)+(0,4mm)$);
\coordinate (MPC_east) at ($(MPC.east)+(0mm,0)$);
\coordinate (system_class_west) at ($(system_class.west)+(0mm,0)$);
\coordinate (top) at ($(stage_cost_1_top.south)+(0,1mm)$);
\coordinate (left)  at ($(system_class.west)+(-2mm,0)$);
\coordinate (right)  at ($(MPC.east)+(2mm,0)$);
\begin{scope}[on background layer]
    \node[
        background2,
        fit={
            (system_class)
            (dyn_ext)
            (weights)
            (priv)
            (stage_cost_1)
            (stage_cost_2)
            (top)
            (left)
            (right)
            (MPC)
        },
        inner xsep=0.25cm,
        inner ysep=0.25cm
    ] (dyn_background) {};
\end{scope}
\begin{scope}[on background layer]
   \node[background, fit=(system_class)(weights)(priv)(stage_cost_1)(stage_cost_1_top)(MPC), inner sep=0.15cm
   ] (kin_background) {};
   \end{scope}
%
%
\coordinate (CurvedArrowCenter)
    at ($(dyn_ext.east)+(0mm,12mm)$);
\draw[
    black,
    line width=0.2mm,
    thick,
    -{Triangle[
        length=3mm,
        width=2.2mm
    ]},
    shorten >=-2mm
]
($(CurvedArrowCenter)+(270:4mm)$)
arc[
    start angle=270,
    end angle=450,
    radius=4mm
];
\node[anchor=west] at ($(CurvedArrowCenter.east)+(3.5mm,-2.5mm)$)
    {\normalsize{treat kinematic subsystem}};

\pgfresetboundingbox
\useasboundingbox
    (dyn_background.south west)
    rectangle
    (dyn_background.north east);
\end{tikzpicture}%
        }%
    }
    \caption{Schematic workflow of the systematic design procedure for the kinematic system with optional dynamic extension, leading to a stabilizing MPC scheme. For Step 3, the parameter $d$ is chosen as in our simulation in Section~\ref{sec:example}, but any positive real number may be chosen, and weighting coefficients may also be used for tuning, cf.~\eqref{eq:stage_cost}.}
    \label{fig:systematic_design}
\end{figure*}

\subsection{Systematic Design Procedure for Kinematic Systems}\label{subsec:Systematic Design Procedure for Kinematic Systems}
In the following, we delineate the systematic design procedure for kinematic systems of the form~\eqref{eq:kinematics_nonholonomic_driftless} under Assumption~\ref{ass:kin_sys_setpoint}.
The procedure consists of three main steps, as illustrated in Fig.~\ref{fig:systematic_design}.

\noindent\textbf{Step 1: Nonholonomic weights}.
We determine parameters that characterize the nonholonomic controllability structure of the system, which we refer to as the nonholonomic weights~$\bm{w}(\p)\in\N^{n_q}$ of the system for the setpoint $\p\in\R^{n_q}$.
More precisely, weight $w_i=w_i(\p)$ is defined as
\begin{align}\label{eq:weights}
    w_i \coloneqq \min \{ s\in \Z_{1:r} \, \vert \, \dim(\Delta^{s}(\p))\geq i \},
\end{align}
i.e., weight~$w_i$ is the smallest integer~$s$ such that the distribution~$\Delta^s (\p)$ is at least of dimension~$i$~\cite{jean_control_2014}. 

\noindent\textbf{Step 2: Privileged coordinates}. In order to define the privileged coordinates, we first state some preliminary definitions.

For a locally smooth function~$f\! :\R^{n_q}\rightarrow\R$, its Lie derivative along a vector field~$\bm{g}:\R^{n_q}\rightarrow\R^{n_q}$ is defined as~$\mathcal{L}_{\bm{g}} f~\coloneqq~\frac{\partial f}{\partial \bm{q}} \bm{g}$. For $b \in \mathbb{N}_0$, repeated Lie derivatives along the vector field~$\bm{g}$ are denoted by
\begin{align*}
    \mathcal{L}_{\bm{g}}^{0} f &\coloneqq f,\\ 
    \mathcal{L}_{\bm{g}}^{b} f
    &\coloneqq
    \mathcal{L}_{\bm{g}}
    \bigl(
        \mathcal{L}_{\bm{g}}^{b-1} f
    \bigr),
    \quad b \geq 1.
\end{align*}

For systems of the form~\eqref{eq:kinematics_nonholonomic_driftless}, the Lie derivative $\mathcal{L}_{\bm{g}_j} f$, along a control vector field~$\bm{g}_j$, $j\in\mathbb{Z}_{1:n_v}$, is referred to as the first-order nonholonomic derivative of~$f$~\cite[p.~14]{jean_control_2014}.
For two control vector fields~$\bm{g}_i, \, \bm{g}_j$ of the system, this concept can be extended to second-order nonholonomic derivatives, e.g.,\[\mathcal{L}_{\bm{g}_i}(\mathcal{L}_{\bm{g}_j} f) = \frac{\partial (\mathcal{L}_{\bm{g}_j }f)}{\partial \bm{q}} \bm{g}_i,\] and higher orders follow analogously.
Then, for the given system, the nonholonomic order $\operatorname{ord}_{\p}(f)$ of the function~$f$ at point $\p\in\mathbb{R}^{n_q}$ is defined as either the largest integer~$k\in\N$ for which all nonholonomic derivatives of~$f$ that are of order smaller than $k$ vanish at~$\p$, or it is infinite if all nonholonomic derivatives of~$f$ vanish at~$\p$. 
Assumption~\ref{ass:kin_sys_setpoint} implies the existence of a system of local coordinates~$\bm{z}=\z(\q;\p)\in\mathbb{R}^{n_q}$ whose coordinate functions~$z_i$ 
satisfy $\operatorname{ord}_{\p}(z_i)=w_i$ for all $i\in\Z_{1:n_q}$ \cite{jean_control_2014}. 
Such coordinates~$\bm{z}$ are called privileged coordinates.
Privileged coordinates can, for instance, be constructed by \emph{Bella\"iche's algorithm}, as presented in~\cite[p. 23]{jean_control_2014}, which has also been employed in~\cite{RoseEbel22}.

We present Bella\"iche's algorithm as Algorithm~\ref{alg:AlgorithmBellaiche} in Appendix~\ref{subsec:BA}, in the notation adopted throughout this paper.
Note that, for many mobile robotic systems, lines~\ref{algStep:forLoopStart}--\ref{algStep:forLoopEnd} simplify considerably, since all required polynomial functions $h_{j,k}$ vanish.

\noindent\textbf{Step 3: Stage cost}. Considering the control system $\dot{\q}=\bm{G}(\bm{q})\bm{v}=:\bm{f}(\q,\bm{v})$, we set the stage costs for the OCP~\eqref{eq:MPC_cont} in terms of privileged coordinates,
\begin{align}
\label{eq:stage_cost}
        \ell(\bm{q},\bm{v}) & = \sum_{i=1}^{n_q} Q_{i} \lvert z_i(\bm{q};\bm{p}) \rvert^{\frac{d}{w_i}}+\sum_{j=1}^{n_v} R_{j} \lvert v_j\rvert^{d} 
\end{align}
with tuning
parameter $d\in(0,\infty)$ and weighting coefficients $Q_i,R_j>0$, $i\in\Z_{1:n_q}$, $j\in\Z_{1:n_v}$. Common choices for the tuning parameter $d$ 
are $d=\prod_{i=1}^{n_q}w_i$ (see~Section~\ref{sec:example}) or $d=2\prod_{i=1}^{n_q}w_i$~\cite{ebel_note_2023}, see also~\cite[Sec.~4.3]{Rosenfelder26}.

This completes the construction of the MPC stage costs for driftless nonholonomic kinematic systems. 
It remains to be shown that the resulting MPC closed loop locally asymptotically stabilizes the desired setpoint.

\subsection{Stability of MPC for Kinematic Systems}\label{subsec:KinMPCStab}
The following theorem establishes the stability of the MPC scheme for the kinematic system with stage costs~\eqref{eq:stage_cost}.
\begin{theorem}
\label{th:MPC}
    Consider a driftless control-affine system with kinematics~\eqref{eq:kinematics_nonholonomic_driftless}, setpoint $\p\in\R^{n_q}$, and let Assumption~\ref{ass:kin_sys_setpoint} hold. Further, let $\bm{w} = \bm{w}(\bm{p})\in \mathbb{N}^{n_q}$ be the corresponding nonholonomic weights, and $\bm{z}(\bm{q};\bm{p})$ denote a system of privileged coordinates at~$\bm{p}$.
    Consider the MPC controller with the OCP~\eqref{eq:MPC_cont} and control system $\dot{\q}=\bm{f}(\q,\bm{v})$. 
    Moreover, let the stage costs be given by~\eqref{eq:stage_cost}. 
    Then, for every sampling time~$\delta_t > 0$, there exists a sufficiently large prediction horizon~$T>0$ such that $\bm{p}$ is locally asymptotically stable  in the sense of Definition~\ref{def:asymptotic_stability} for the loop closed by the MPC controller.
\end{theorem}

In order to 
prove this theorem, we first establish definitions for homogeneous systems as well as homogeneous approximations.
\begin{definition}[Homogeneous control system \cite{Gruene06}]\label{def:homogeneousSystem}
A system of the form
\begin{align*}
    \dot{\bm{x}}(t)&= \bm{f}(\x(t),\bu(t))
\end{align*}
is a homogeneous system with homogeneity parameters $\bm{r}=[r_1\ \cdots\ r_{n_x}]\in(0,+\infty)^{n_x},\ \bm{s}=[s_1\ \cdots\ s_{n_u}]\in(0,+\infty)^{n_u}$, and $\tau\in(-\min_{i}r_i,+\infty)$ if
\begin{align}
\label{eq:homogeneity_condition_definition}    \bm{f}(\Lambda_\alpha\bm{x},\Delta_{\alpha}\bu)&=\alpha^{\tau}\Lambda_{\alpha} \bm{f}(\x,\bu)
\end{align}
holds for all $(\bm{x},\bm{u})\in\R^{n_x}\times\R^{n_u}$ and all $\alpha\in\R_{\geq 0}$, with 
\begin{align*}
    \bm{\Lambda}_\alpha&=\operatorname{diag}(\alpha^{r_1},\ldots,\alpha^{r_{n_x}})\in\R^{n_x\times n_x},\\
    \bm{\Delta}_\alpha&=\operatorname{diag}(\alpha^{s_1},\ldots,\alpha^{s_{n_u}})\in\R^{{n_u}\times {n_u}}.
\end{align*}
We denote such a system as $(\bm{r},\bm{s},\tau)$-homogeneous.
\end{definition}
\begin{definition}[Homogeneous approximation~\cite{coron_model_2019}]
\label{def:homogeneous_approximation}
Let $\bm{h}(\x,\bu)$ describe the kinematics of a $(\bm{r},\bm{s},\tau)$-homogeneous system according to Definition~\ref{def:homogeneousSystem}.
Then, $\bm{h}$ is said to be a homogeneous approximation of $\bm{f}(\x,\bu)$ near the origin if the relation
\begin{align}
\label{eq:homogeneousApproximationRelation}
\bm{f}(\bm{x},\bm{u}) = \bm{h}(\bm{x},\bm{u}) + \bm{R}(\bm{x},\bm{u})
\end{align}
holds with
\begin{align}
\label{eq:homogeneousApproximationResiduum}
|R_{i}(\bm{\Lambda}_{\alpha}\bm{x},\bm{\Delta}_{\alpha}\bm{u})| \leq M \alpha^{r_{i}+\tau+\eta} \quad \forall\,i \in \mathbb{Z}_{1:n_x}
\end{align}
for all $\x,\bu$ with $\left\lVert \bm{x}\right\rVert \leq \rho, \left\lVert \bm{u}\right\rVert \leq \rho$ and all $\alpha \in (0,1]$ with $\rho,M,\eta >0$.
\end{definition}
Regarding a general nonlinear control system $\dot{\bm{x}} = \bm{f}(\bm{x}, \bm{u})$, as presented in Sec.~\ref{sec:Problem Formulation},  we assume that a $(\bm{r},\bm{s},0)$-homogeneous approximation (cf. Definition~\ref{def:homogeneous_approximation}) exists, which is locally asymptotically null controllable to the origin, cf. \cite[Def.~3.2]{coron_model_2019}. 
Then, the origin is stabilized for the MPC-closed-loop with the OCP~\eqref{eq:MPC_cont} (cf. Definition~\ref{def:asymptotic_stability}) with stage costs~\eqref{eq:stage_cost}. Further, for the stage costs, the tuning parameter $d\in(0,\infty)$ as well as the positive weighting coefficients can be chosen arbitrarily.~\cite[Th.~4.4, Rem.~4.7]{coron_model_2019}.
Since we exclusively consider driftless systems, we simply shift the desired setpoint to the {origin}. 


\begin{proposition}
\label{prop:homogeneous_approx_kinematic}
    For the kinematic system~\eqref{eq:kinematics_nonholonomic_driftless} under Assumption~\ref{ass:kin_sys_setpoint}, expressed in privileged coordinates, there exists a $(\bm{w},\bm{1}_{n_v},0)$-homogeneous system $\dot{\z}=\bar{\bm{f}}(\z,\bm{v})$, which is a homogeneous approximation of the kinematic system near the origin. Furthermore, this homogeneous approximation is locally null controllable to the origin.
\end{proposition}
\begin{proof}
    See end of Section~\ref{sec:Homogeneous Approximation for Kinematic Systems}.
\end{proof}
\begin{proof}[Proof of Theorem~\ref{th:MPC}]
    We consider both the system and its homogeneous approximation (according to Proposition~\ref{prop:homogeneous_approx_kinematic}) in privileged coordinates.
    It follows from Proposition~\ref{prop:homogeneous_approx_kinematic} and \cite[Th.~4.4, Rem.~4.7]{coron_model_2019} that the origin (in privileged coordinates) is locally asymptotically stable for the MPC closed-loop with OCP~\eqref{eq:MPC_cont} 
    for stage costs~\eqref{eq:stage_cost}.
    
    Since the coordinate transformation into privileged coordinates is a (local) diffeomorphism~\cite{jean_control_2014}, the stabilization of the setpoint in $\z$-coordinates implies the stabilization in the original coordinates.
\end{proof}
The key component of the tailored MPC controller is the mixed-exponents stage cost~\eqref{eq:stage_cost} whose structure reflects the underlying distance measure induced by the nonholonomic geometry of the system.
\begin{rem}
\label{rem:sub_Riemannian_distance}
We regard the kinematic system~\eqref{eq:kinematics_nonholonomic_driftless}, under the assumption that $\operatorname{rank}(\bm{G}(\bm{q}))=n_v$ for all $\bm{q}\in\R^{n_q}$, and where the velocity $\bm{v}$ acts as the control input. Then, the sub-Riemannian distance  between two points $\bm{q}_1,\bm{q}_2\in\R^{n_q}$ induced by this system is defined as
\begin{equation}
d_{\textnormal{R}}(\bm{q}_1,\bm{q}_2) \coloneqq \inf_{T>0} \inf_{\bm{v}(\cdot)} \int_0^T \|\bm{v}(t)\|\, \mathrm{d}t,
\end{equation}
where the infimum is taken over all $T>0$ and all controls $\bm{v}(\cdot)\in L^{1}([0,T],\R^{m})$ steering the system~\eqref{eq:kinematics_nonholonomic_driftless} from $\bm{q}_1$ to $\bm{q}_2$~\cite{jean_control_2014}.

For $d \geq 1$, the $d$-th root of the state-dependent part of the MPC stage costs~\eqref{eq:stage_cost}, with weighting coefficients equal to $1$, is locally equivalent to the sub-Riemannian distance, i.e.,
\begin{align*}
\left( \sum_{i=1}^{n_q}\lvert z_i(\bm{q};\p) \rvert^{\frac{d}{w_i}} \right) ^{\frac{1}{d}} \sim d_{\textnormal{R}}(\bm{q},\p).
\end{align*}
This local equivalence is justified in App.~\ref{app:equivalence_subRiem_cost}.
For $d=1$, this term is equivalent to the so-called pseudo-norm of the privileged coordinates at $\p$, as defined in~\cite[p.~21]{jean_control_2014}, see also~\cite[p.~171]{Polyakov2025}. 
\end{rem}
\subsection{Dynamic Extension: MPC Design and Stability}\label{subsec:DynExtMPCStab}

We now go beyond the driftless case and cover nonholonomic systems with a dynamic extension, where we can state an analogous stability theorem. 
The design procedure 
only differs in the inclusion
of the force-level input values to the stage costs. 
\begin{theorem}\label{thm:MPC_extension} 
Consider the nonholonomic kinematic system with dynamic extension~\eqref{eq:dynamic_extension_system}, with the state~$\bm{x}=[\bm{q}^\top \ \bm{v}^\top
]^\top \in \mathbb{R}^{n_x}$. Let $\p$ be a setpoint for the kinematic subsystem~\eqref{eq:kinematics_nonholonomic_driftless}, and let Assumption~\ref{ass:kin_sys_setpoint} hold.
Then, let the setpoint for the dynamic system~\eqref{eq:dynamic_extension_system} read~$\widetilde{\bm{p}}=[
    \bm{p}^\top \ \bm{0}_{n_v}^\top
]^\top\in\mathbb{R}^{n_x}$.
Let $\bm{w}(\p) \in \mathbb{N}^{n_q}$ be the nonholonomic weights of the kinematic subsystem~\eqref{eq:kinematics_nonholonomic_driftless} and let $\bm{z}(\bm{q};\bm{p})$ denote a corresponding system of privileged coordinates.
Consider an MPC controller with the OCP~\eqref{eq:MPC_cont} with 
\begin{align*}
  \bm{f}(\x,\bu)=\begin{bmatrix}
        \bm{G} (\bm{q})\bm{v} \\ \bm{M}^{-1}(\bm{q})\bm{B}{(\bm{q})}\bm{u} 
    \end{bmatrix}
\end{align*}
and stage costs
\begin{align}\label{eq:stage_cost_extended}
    \ell(\bm{x},\bm{u}) & = \ell(\bm{q},\bm{v})+\sum_{j=1}^{n_u} \bar{R}_{j} \lvert u_j\rvert^{d}\\
    &=\sum_{i=1}^{n_q} Q_{i} \lvert z_{i} (\bm{q};\bm{p})\rvert^{\frac{d}{w_i}} +\sum_{j=1}^{n_v} R_{j} \lvert v_j \rvert^{d} + 
    \sum_{j=1}^{n_v} \bar{R}_{j} \lvert u_j\rvert^{d} \nonumber
\end{align}
with tuning parameter~$d\in(0,\infty)$ and weights $Q_{i}>0$, $i\in\mathbb{Z}_{1:n_q}$, and $R_{j},\bar{R}_{j} > 0$, $j\in\mathbb{Z}_{1:n_v}$.
Then, for every sampling time $\delta_t > 0$, there exists a sufficiently large prediction horizon~$T$ such that $\widetilde{\bm{p}}$ is locally asymptotically stable  in the sense of Definition~\ref{def:asymptotic_stability} for the loop closed by the MPC controller. 
\end{theorem}
\begin{proposition}
\label{prop:homogeneous_approx_w_dyn_ext}
Consider the nonholonomic kinematic system with dynamic extension~\eqref{eq:dynamic_extension_system}, with the state~$\bm{x}=[\bm{q}^\top \ \bm{v}^\top
]^\top \in \mathbb{R}^{n_x}$. Let $\p$ be a setpoint for the kinematic subsystem~\eqref{eq:kinematics_nonholonomic_driftless}, and let Assumption~\ref{ass:kin_sys_setpoint} hold. 
Then, there exists a $([\bm{1}_{n_v}^\top\,\bm{w}^\top]^\top,\bm{1}_{n_v},0)$-homogeneous system that is a homogeneous approximation of the dynamic system near the origin, expressed in the coordinates $[\bm{v}^\top \ \bm{z}^\top]^\top$, where $\z=\z(\bm{q};\p)$ are privileged coordinates of the kinematic subsystem. 
Moreover, the approximation is locally null controllable to the origin. 
\end{proposition}
\begin{proof}
    See end of Section~\ref{sec:Homogeneous Approximation for Kinematic Systems with Dynamic Extension}.
\end{proof}
\begin{proof}[Proof of Theorem~\ref{thm:MPC_extension}]
    The proof follows analogously to that of Theorem~\ref{th:MPC}. Proposition~\ref{prop:homogeneous_approx_w_dyn_ext} and \cite[Th.~4.4, Rem.~4.7]{coron_model_2019} yield local asymptotic stability of the origin for state $\bm{x}=[\bm{q}^\top \ \bm{v}^\top
]^\top \in \mathbb{R}^{n_x}$ when applying the stage costs~\eqref{eq:stage_cost_extended}. 
\end{proof}

\section{Design Procedure: Example}\label{sec:example}
The following example illustrates each step of the proposed design procedure, from the computation of the nonholonomic weights to the construction of the tailored stage cost.
To illustrate the proposed design procedure, 
we consider the articulated vehicle shown in Fig.~\ref{fig:sketch_artVeh}, i.e., a vehicle composed of two rigid bodies connected by an actuated joint. 

\subsection{Kinematic System}
    A kinematic model of the articulated vehicle 
    is given by
    \begin{align}
    \label{equ:kin_artVeh}
    \dot{\bm{q}} = 
    \underbrace{
    \begin{bmatrix} 
        \cos q_3 \\   
        \sin q_3 \\
        \frac{\sin q_4}{l_1 \cos q_4+l_2} \\
        0
    \end{bmatrix}}_{\displaystyle \bm{g}_1(\bm{q})}
    v_{1} + 
    \underbrace{
    \begin{bmatrix} 
        0 \\   
        0 \\
        \frac{l_2}{l_1 \cos q_4+l_2} \\
        1
    \end{bmatrix}}_{\displaystyle \bm{g}_2(\bm{q})}
    v_{2}
    \eqqcolon \bm{G}(\bm{q})\bm{v},
    \end{align}
    where the state~$\bm{q}\in\mathbb{R}^4$ comprises the planar position~$\begin{bmatrix}q_1 & q_2\end{bmatrix}^\top$ of the front vehicle body, its orientation~$q_3$, and the articulation angle~$q_4$, which is defined as the difference between the orientations of the front and rear vehicle bodies, cf.\ Fig.~\ref{fig:sketch_artVeh}. 
    The velocity-level input~$\bm{v}\in\mathbb{V} \subset \mathbb{R}^2$ consists of the translational velocity of the front vehicle body and the angular velocity of the articulation joint, respectively.
    For shortness of notation, we consider the specific vehicle geometry
    $l_1=l_2=\unit[1]{m}$ and subsequently use only the numerical values in SI units, i.e., $l_1=l_2=1$.
    The distributions are derived as $\Delta^{1} = \textnormal{span} \lbrace \bm{g}_1(\bm{q}),\bm{g}_2(\bm{q}) \rbrace$,
    $\Delta^{2} = \textnormal{span} \lbrace \bm{g}_1(\bm{q}),\bm{g}_2(\bm{q}), \bm{g}_3(\bm{q}) \rbrace$,
    and $\Delta^{3} = \textnormal{span} \lbrace \bm{g}_1(\bm{q}),\bm{g}_2(\bm{q}), \bm{g}_3(\bm{q}), \bm{g}_4(\bm{q}) \rbrace$,
    where 
    \begin{align*}
        \bm{g}_3(\bm{q}) &= [\bm{g}_1, \bm{g}_2] (\bm{q}) = a \begin{bmatrix}
         \sin q_3&
    -\cos q_3&
    -1 &
    0 \end{bmatrix}^\top,
    \\
    \bm{g}_4(\bm{q}) &= [\bm{g}_1, \bm{g}_3] (\bm{q}) =  a^2 \begin{bmatrix}
        -\sin(q_3-q_4)-\sin q_3\\
        \cos(q_3-q_4)+\cos q_3\\
0 \\0 \end{bmatrix}
    \end{align*}
    with $a = 1/ \left( \cos q_4 + 1 \right)$.    
    At $\bm{p}=\bm{0}$, the vector fields are locally smooth and analytic, excluding the singular configurations
    $q_4=(2k + 1)\pi,\ k\in \Z$. 
    Similarly, the ranks of the distributions are constant around $\bm{p}=\bm{0}$, i.e.,
    $\dim(\Delta^1)=2$, $\dim(\Delta^2)=3$, and $\dim(\Delta^3)=4$.
    Thus, $\bm{p}$ is regular and the LARC is fulfilled. The nonholonomic weights follow as
    $\bm{w} = \begin{bmatrix}1 & 1 & 2 & 3\end{bmatrix}^\top$.
    
    The coordinates~$\bm{q}$ are not privileged at the setpoint $\bm{p} = \bm{0}$, as, for instance, $\operatorname{ord}_{\bm{0}}(q_4)=1 \neq w_4$.
    However, the coordinates $\bm{z} = [q_1 \ \ q_4 \ \ q_4-2 q_3 \ \ 2 (q_2- q_3) +q_4]^\top$ obtained with Bella\"iche's Algorithm~\cite[p. 23]{jean_control_2014} form a system of privileged coordinates.
    In these coordinates, the kinematics read
    \begin{align}
    \label{equ:kin_artVeh_z}
    \dot{\bm{z}}\! =\! 
    \underbrace{
    \begin{bmatrix} 
        \cos (\frac{1}{2}(z_2-z_3)) \\   
        0 \\
        -2 \widetilde{a} \sin(z_2) \\
        2 \left(\sin (\frac{1}{2}(z_2-z_3)) - \widetilde{a} \sin(z_2)\right)
    \end{bmatrix}}_{\displaystyle \bm{Z}_1(\bm{z})}\!
    v_{1}\! + \!
    \underbrace{
    \begin{bmatrix} 
        0 \\   
        1 \\
        1-2 \widetilde{a} \\
        1-2 \widetilde{a}
    \end{bmatrix}}_{\displaystyle \bm{Z}_2(\bm{z})}\!
    v_{2}
    \end{align}
    with $\widetilde{a} = 1/ \left( \cos(z_2)+1 \right)$.
    Indeed, $\operatorname{ord}_{\bm{0}}(z_i)=w_i$ holds for all $i\in\Z_{1:n_q}$.
    As an example, consider the nonholonomic order of $z_3$ at the origin.
    The first-order nonholonomic derivatives $\mathcal{L}_{\bm{Z}_1} z_3 =  -2 \widetilde{a} \sin(z_2)$ and $\mathcal{L}_{\bm{Z}_2} z_3 = 1-2 \widetilde{a}$ vanish at $\bm{p}=\bm{0}$, while the second-order derivative
    $\mathcal{L}_{\bm{Z}_2}(\mathcal{L}_{\bm{Z}_1} z_3) = -2 \widetilde{a}$
    is nonzero at $\bm{p} = \bm{0}$.
    Consequently, $\operatorname{ord}_{\bm{0}}(z_3)=w_3$ holds.


%
Using the nonholonomic weights and the privileged coordinates, the stage costs as described in Theorem~\ref{th:MPC} read
\begin{align*}
    \ell(\bm{q},\bm{v}) & =
    Q_{1} \lvert q_1 \rvert ^{d} + Q_{2} \lvert q_4 \rvert ^{d} + Q_{3} \lvert q_4-2 q_3 \rvert ^{\frac{d}{2}} \\
    &+ Q_{4} \lvert 2 (q_2-q_3)+q_4 \rvert ^{\frac{d}{3}}
    + \sum_{j=1}^{2} R_{j}\lvert v_j\rvert^{d}
\end{align*}
with $d\in(0,\infty)$, $Q_{i}>0$, $i\in\mathbb{Z}_{1:4}$, and $R_{j} > 0$, $j\in\mathbb{Z}_{1:2}$.
%
\begin{figure}
    \centering
    \tikzset{
  COSline/.style = {line width = 0.4mm, ->, , >=stealth},
  CoordinateArrow/.style = {line width = 0.3mm, ->, blue},
  CoordinateName/.style = {blue},
  Notation/.style = {gray},
}

\begin{tikzpicture}[>=triangle 45]
\def\lCOS{1}
\def\xP{3.5}
\def\yP{1.8}
\def\thetaa{30}
\def\thetaaT{5}

\def\Laxle{1.1} 
\def\b{1.35*\Laxle} 
\def\Ltotal{1.25} 
\def\Lrect{0.8*\Ltotal}
\def\shiftMiddleRobot{-0.85*\Lrect}

\def\Rwheel{0.2*\Laxle} 
\def\bwheel{0.6*\Rwheel} 

\def\ax{1.2}
\def\ay{0}
\def\bx{1.2}
\def\by{0}
\def\widthTrailer{1.7}
\def\heightTrailer{0.975*\b}
\def\LaxleTrailer{0.7*\heightTrailer} 
\def\shiftMiddleTrailer{0}
\pgfmathsetmacro{\widthLinkingRobot}{\ax+\shiftMiddleRobot} 

\def\opacChassis{0.5} 
\def\colorChassis{cyan}

\def\colorChassisTrailer{yellow}
 
\def\widthLinkingTrailer{\bx-\shiftMiddleTrailer-0.5*\widthTrailer}
\def\heightLinkingTrailer{0.4}

\draw[gray, very thin] (0,0) -- ++(1+\xP,0); 
\draw[gray, very thin] (0,0) -- ++(0, 1+\yP); 
\draw[COSline] (0,0) -- ++(1,0); 
\draw[COSline] (0,0) -- ++(0,1); 
\filldraw[fill=black, draw=none] (0,0) circle (0.05) node[below, left] {$\mathcal{K}_\textnormal{I}$};

\draw[dashed] (\xP,\yP) -- (\xP,0) node[below] {$q_1$};
\draw[dashed] (\xP,\yP) -- (0,\yP) node[left] {$q_2$};

\begin{scope}[shift={(\xP,\yP)}, rotate = \thetaa]
	\coordinate (P) at (0,0);
	\begin{scope}[shift={(-\ax,\ay)}]
		\coordinate (Q) at (0,0);
	\end{scope}	
	\begin{scope}[shift={(\shiftMiddleRobot,-0.5*\b)}]
		\fill[fill=\colorChassis, opacity = \opacChassis] (0,0) rectangle ++(\Lrect,\b);
		\fill[fill=\colorChassis, opacity = \opacChassis] (\Lrect,0) arc[start angle=-90, end angle=90, radius=0.5*\b];

		\fill[fill=\colorChassis, opacity = \opacChassis] (-\widthLinkingRobot,\b/2-\heightLinkingTrailer/2) rectangle ++(\widthLinkingRobot,\heightLinkingTrailer);

		\fill[fill=\colorChassis, opacity = \opacChassis]
			(-\widthLinkingRobot,\b/2-\heightLinkingTrailer/2)
			arc[start angle=270, end angle=90, radius=0.5*\heightLinkingTrailer];
	\end{scope}
		
	\draw[line width = 0.3mm] (P) -- (Q) node[above, pos=0.7, xshift=-0.3mm, yshift=-0.6mm] {$l_1$};

	\draw[] (0,0.5*\Laxle) -- (0,-0.5*\Laxle);
	\begin{scope}[shift={(0,0.5*\Laxle)}]
		\filldraw[fill=black] (-\Rwheel,-0.5*\bwheel) rectangle ++(2*\Rwheel,\bwheel); 
	\end{scope}
	\begin{scope}[shift={(0,-0.5*\Laxle)}]
		\filldraw[fill=black] (-\Rwheel,-0.5*\bwheel) rectangle ++(2*\Rwheel,\bwheel); 
	\end{scope}

	\begin{scope}[shift={(Q)}, rotate = (\thetaaT - \thetaa), shift={(-\bx,0)}]
		\begin{scope}[shift={(-0.5*\widthTrailer,-0.5*\heightTrailer)}, shift={(\shiftMiddleTrailer,0)}] 
			\fill[fill=\colorChassisTrailer, opacity = \opacChassis] (0,0) rectangle ++(\widthTrailer,\heightTrailer);
			\fill[fill=\colorChassisTrailer, opacity = \opacChassis] (\widthTrailer,0.5*\heightTrailer-0.5*\heightLinkingTrailer) rectangle ++(\widthLinkingTrailer,\heightLinkingTrailer);
			\fill[fill=\colorChassisTrailer, opacity = \opacChassis]
				(\widthTrailer+\widthLinkingTrailer,0.5*\heightTrailer-0.5*\heightLinkingTrailer)
				arc[start angle=-90, end angle=90, radius=0.5*\heightLinkingTrailer];
		\end{scope}

		\coordinate (Pt) at (0,0);

		\draw[] (0,0.5*\LaxleTrailer) -- (0,-0.5*\LaxleTrailer);
		\begin{scope}[shift={(0,0.5*\LaxleTrailer)}]
			\filldraw[fill=black] (-\Rwheel,-0.5*\bwheel) rectangle ++(2*\Rwheel,\bwheel); 
		\end{scope}
		\begin{scope}[shift={(0,-0.5*\LaxleTrailer)}]
			\filldraw[fill=black] (-\Rwheel,-0.5*\bwheel) rectangle ++(2*\Rwheel,\bwheel); 
		\end{scope}

		\begin{scope}
			\coordinate (QB) at (0,0);
			\coordinate (QBb) at (1,0);

            \draw[line width = 0.3mm] (QB) -- (Q) node[above, pos=0.5] {$l_2$};;
		\end{scope}

	\end{scope}

	\begin{scope}[rotate = -\thetaa]
		\draw[Notation, dashed] (0,0) -- ++(1.3,0);
		\draw[->] (1.2,0) arc[start angle=0, end angle=\thetaa, radius=1.2] node[pos=0.4, right] {$q_3$};
	\end{scope}
	\draw[Notation, dashed] (0,0) -- ++(1.1,0);


	\begin{scope}[shift={(Pt)}, rotate = \thetaaT - \thetaa]
		\draw[Notation, dashed] (Q) -- ++(1,0);
			\draw[->] (\bx+0.65,0) arc[start angle=0, end angle=\thetaa- \thetaaT, radius=0.65] node[pos=0.4, right, xshift=-0.1mm, yshift=1mm] {$q_4$};
	\end{scope}


	\filldraw[fill=black, draw=none] (0,0) circle (0.08);
	\filldraw[fill=black, draw=none] (Q) circle (0.08);
	\filldraw[fill=white, draw=none] (Q) circle (0.05);
	\filldraw[fill=black, draw=none] (QB) circle (0.08);

	\filldraw[fill=black, draw=none] (Pt) circle (0.08);


\end{scope}

\begin{scope}
    \def\hitchWidth{0.25}
    \draw[line width=0.25mm, fill=white] (Q) circle (0.11);
    \filldraw[fill=black, draw=none] (Q) circle (0.035);
\end{scope}

\end{tikzpicture}
    \caption{Schematics of the articulated vehicle used as an example for a kinematic system with dynamic extension.}
    \label{fig:sketch_artVeh}
\end{figure}
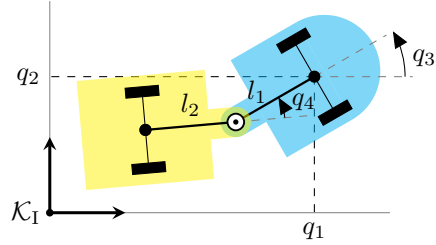
\subsection{Dynamic Extension}

    While the kinematic model~\eqref{equ:kin_artVeh} describes the vehicle motion at the velocity level, the dynamic extension additionally accounts for the acting forces and torques as well as the inertial properties of the rigid bodies. 
    For each axle, the two wheels are collapsed into a single representative wheel located at the axle center.
    The resulting dynamic extension is given by~\eqref{eq:dynamic_extension_system}, with the kinematic part $\dot{\bm{q}} =\bm{G} (\bm{q} )\bm{v}$ as defined in~\eqref{equ:kin_artVeh}, $\bm{B} = \diag(b_{11},1)$, and
    \begin{align*}
        &\bm{M} (q_4) \\
        &= a^2 \begin{bmatrix}
            \mu_{11,0} + \mu_{11,c}\,\text{c}_4 &
            \text{s}_4 \left( \mu_{12,s}+\mu_{12,sc}\,\text{c}_4   \right) \\
            \text{s}_4 \left( \mu_{12,s}+\mu_{12,sc}\,\text{c}_4   \right) &
            \mu_{22,0} + \mu_{22,cc}\,\text{c}_4^2 
        \end{bmatrix},
    \end{align*}
    where $a = 1/ \left( \cos q_4 + 1 \right)$, $\text{c}_4 \coloneqq \cos q_4$, and \\
    $\text{s}_4 \coloneqq \sin q_4$.
    Moreover, $\bm{M} (q_4) = \bm{M}(q_4)^\top \succ 0$, and $b_{11}, \, \mu_{\bullet} \in\mathbb{R}$ are constants depending on the geometric and inertial parameters of the system.
    The extended state includes the velocities $\bm{v}$, whereas the inputs~$u_1$ and $u_2$ denote the front-wheel torque and the articulation torque, respectively.
    The stage costs given in Theorem~\ref{thm:MPC_extension} then read
    \begin{align}\label{eq:stage_cost_extended_robot}
        \ell(\bm{q},\bm{v},\bm{u}) =  
        Q_{1}& \lvert q_1 \rvert^{d} + Q_{2} \lvert q_4 \rvert ^{d} + Q_{3} \lvert q_4-2 q_3 \rvert^{\frac{d}{2}}  \nonumber \\
        {} & + Q_{4} \lvert 2 (q_2-q_3)+q_4 \rvert ^{\frac{d}{3}}\nonumber\\
        {}&+\sum_{j=1}^{2}R_{j} \lvert v_j \rvert^{d} + 
    \sum_{j=1}^{2} \bar{R}_{j}\lvert u_j\rvert^{d} 
    \end{align}
    with arbitrary~$d\in(0,\infty)$ and positive weighting coefficients $Q_{i}>0$, $i\in\mathbb{Z}_{1:n_q}$, and $R_{j}, \bar{R}_{j} > 0$, $j\in\mathbb{Z}_{1:n_u}$.
    %

\begin{figure*}[t]
    \centering
    \raisebox{-0.08cm}{\begin{tikzpicture}
\def\lw{0.4mm}
\def\hCLplot{4cm}
\pgfplotsset{
    PlotOSavg/.style={
        line width = 0.4mm
    }
}

\def\lCOS{1}
\def\xP{-2*\ax}
\def\yP{-1*\ax}
\def\thetaa{0}
\def\thetaaT{0}

\def\Laxle{1.1} 
\def\b{1.35*\Laxle} 
\def\Ltotal{1.25} 
\def\Lrect{0.8*\Ltotal}
\def\shiftMiddleRobot{-0.85*\Lrect}

\def\Rwheel{0.2*\Laxle} 
\def\bwheel{0.6*\Rwheel} 

\def\ax{1.2}
\def\ay{0}
\def\bx{1.2}
\def\by{0}
\def\widthTrailer{1.7}
\def\heightTrailer{0.975*\b}
\def\LaxleTrailer{0.7*\heightTrailer} 
\def\shiftMiddleTrailer{0}
\pgfmathsetmacro{\widthLinkingRobot}{\ax+\shiftMiddleRobot} 

\def\opacChassis{0.5} 
\def\colorChassis{cyan}

\def\colorChassisTrailer{yellow}
 
\def\widthLinkingTrailer{\bx-\shiftMiddleTrailer-0.5*\widthTrailer}
\def\heightLinkingTrailer{0.4}

\begin{axis}[
        name = path,
        anchor=south west,
        height = \hCLplot,
        xmin=-5,
        xmax=1.7,
        ymin=-3,
        ymax=2,
        axis equal image,
        xmajorgrids,
        ymajorgrids,
        xlabel = {$q_1$ in m},
        ylabel = {$q_2$ in m},
		ylabel style={align=center, yshift=-0.5cm},
        legend columns = 1,
        legend style={
            at={(1,1.03)},
            anchor=south east,
            legend cell align=left,
            draw=none,
            fill=none,
            font=\small,
            column sep=0.15cm,
            /tikz/every even column/.append style={column sep=0.15cm}
        }
	]

    \begin{scope}[
        shift={(axis cs:0,0)},
        x={(axis direction cs:1,0)},
        y={(axis direction cs:0,1)},
        scale={1/\ax}, 
        rotate=0
    ]

        \begin{scope}[shift={(\xP,\yP)}, rotate = \thetaa]
        \coordinate (P) at (0,0);
        \begin{scope}[shift={(-\ax,\ay)}]
            \coordinate (Q) at (0,0);
        \end{scope}	
        \begin{scope}[shift={(\shiftMiddleRobot,-0.5*\b)}]
            \fill[fill=\colorChassis, opacity = \opacChassis] (0,0) rectangle ++(\Lrect,\b);
            \fill[fill=\colorChassis, opacity = \opacChassis] (\Lrect,0) arc[start angle=-90, end angle=90, radius=0.5*\b];

            \fill[fill=\colorChassis, opacity = \opacChassis] (-\widthLinkingRobot,\b/2-\heightLinkingTrailer/2) rectangle ++(\widthLinkingRobot,\heightLinkingTrailer);

            \fill[fill=\colorChassis, opacity = \opacChassis]
                (-\widthLinkingRobot,\b/2-\heightLinkingTrailer/2)
                arc[start angle=270, end angle=90, radius=0.5*\heightLinkingTrailer];
        \end{scope}
            
        \draw[line width = 0.3mm] (P) -- (Q) node[above, pos=0.7, xshift=-0.3mm, yshift=-0.6mm] {};

        \draw[] (0,0.5*\Laxle) -- (0,-0.5*\Laxle);
        \begin{scope}[shift={(0,0.5*\Laxle)}]
            \filldraw[fill=black] (-\Rwheel,-0.5*\bwheel) rectangle ++(2*\Rwheel,\bwheel); 
        \end{scope}
        \begin{scope}[shift={(0,-0.5*\Laxle)}]
            \filldraw[fill=black] (-\Rwheel,-0.5*\bwheel) rectangle ++(2*\Rwheel,\bwheel); 
        \end{scope}

        \begin{scope}[shift={(Q)}, rotate = (\thetaaT - \thetaa), shift={(-\bx,0)}]
            \begin{scope}[shift={(-0.5*\widthTrailer,-0.5*\heightTrailer)}, shift={(\shiftMiddleTrailer,0)}] 
                \fill[fill=\colorChassisTrailer, opacity = \opacChassis] (0,0) rectangle ++(\widthTrailer,\heightTrailer);
                \fill[fill=\colorChassisTrailer, opacity = \opacChassis] (\widthTrailer,0.5*\heightTrailer-0.5*\heightLinkingTrailer) rectangle ++(\widthLinkingTrailer,\heightLinkingTrailer);
                \fill[fill=\colorChassisTrailer, opacity = \opacChassis]
                    (\widthTrailer+\widthLinkingTrailer,0.5*\heightTrailer-0.5*\heightLinkingTrailer)
                    arc[start angle=-90, end angle=90, radius=0.5*\heightLinkingTrailer];
            \end{scope}

            \coordinate (Pt) at (0,0);

            \draw[] (0,0.5*\LaxleTrailer) -- (0,-0.5*\LaxleTrailer);
            \begin{scope}[shift={(0,0.5*\LaxleTrailer)}]
                \filldraw[fill=black] (-\Rwheel,-0.5*\bwheel) rectangle ++(2*\Rwheel,\bwheel); 
            \end{scope}
            \begin{scope}[shift={(0,-0.5*\LaxleTrailer)}]
                \filldraw[fill=black] (-\Rwheel,-0.5*\bwheel) rectangle ++(2*\Rwheel,\bwheel); 
            \end{scope}

            \begin{scope}
                \coordinate (QB) at (0,0);
                \coordinate (QBb) at (1,0);

                \draw[line width = 0.3mm] (QB) -- (Q) node[above, pos=0.5] {};;
            \end{scope}

        \end{scope}

        \filldraw[fill=black, draw=none] (0,0) circle (0.08);
        \filldraw[fill=black, draw=none] (Q) circle (0.08);
        \filldraw[fill=white, draw=none] (Q) circle (0.05);
        \filldraw[fill=black, draw=none] (QB) circle (0.08);

        \filldraw[fill=black, draw=none] (Pt) circle (0.08);


    \end{scope}

    \begin{scope}
        \def\hitchWidth{0.25}
        \draw[line width=0.25mm, fill=white] (Q) circle (0.11);
        \filldraw[fill=black, draw=none] (Q) circle (0.035);
    \end{scope}

    \end{scope}

    \addplot[blueUni,PlotOSavg, dashdotted] table[ x expr=\thisrow{x}, y expr=\thisrow{y}] {SimArtVeh/Trajectoryx_x_y_theta_gamma_phi_l_phi_1_t_v_gamma_dot.txt};
    \addlegendentry{front vehicle body}
    \addplot[redUni,PlotOSavg, dotted] table[ x expr=\thisrow{xt}, y expr=\thisrow{yt}] {SimArtVeh/Trajectoryx_xt_yt_thetat.txt};
    \addlegendentry{rear vehicle body}

\end{axis}

\end{tikzpicture}}%
    \begin{tikzpicture}
\def\lw{0.4mm}
\def\hCLplot{4cm}
\pgfplotsset{
    PlotOSavg/.style={
        line width = 0.4mm
    }
}

\begin{axis}[
        name = valueFunction,
        anchor=south west,
        height = \hCLplot,
        width = 0.3\textwidth,
        xmin = 0,
        xmax = 100,
        ymax = 1e3,
        xtick={0,30,60,90},
        xmajorgrids,
        ymajorgrids,
        xlabel = {time in s},
        ylabel = {value function},
		ymode = log,
		ylabel style={align=center},
        legend columns = 1,
        legend cell align=left,
        legend style={
            at={(1,1.03)},
            anchor=south east,
            draw=none,
            fill=none,
            font=\small,
            column sep=0.15cm,
            /tikz/every even column/.append style={column sep=0.15cm}
        }
	]   
    
    \addplot[blueUni,PlotOSavg] table[ x expr=\thisrow{time}, y expr=\thisrow{costs}*1e-7] {SimArtVeh/Trajectorycosts.txt};
    \addlegendentry{mixed-exponents}
    \addplot[greenUni,PlotOSavg] table[ x expr=\thisrow{time}, y expr=\thisrow{costs}*1e-7] {SimArtVeh/Trajectorycosts_quad.txt};
    \addlegendentry{quadratic}
\end{axis}

\end{tikzpicture}%
    \begin{tikzpicture}
\def\lw{0.4mm}
\def\hCLplot{4cm}
\pgfplotsset{
    PlotOSavg/.style={
        line width = 0.4mm
    }
}
\pgfplotsset{
    PlotThin/.style={
        line width = 0.25mm
    }
}

\begin{axis}[
        name = valueFunction,
        anchor=south west,
        height = \hCLplot,
        width = 0.4\textwidth,
        xmin = 0,
        xmax = 100, 
        xtick={0,30,60,90},
        xmajorgrids,
        ymajorgrids,
        xlabel = {time in s},
        ylabel = {inputs in $\unit{Nm}$},
		ylabel style={align=center, yshift=-0.5cm},
        legend columns = 2,
        legend cell align=left,
        legend style={
            at={(1,1.03)},
            anchor=south east,
            draw=none,
            fill=none,
            font=\small,
            column sep=0.15cm,
            /tikz/every even column/.append style={column sep=0.15cm}
        }
	]   
    
    \addplot[blueUni,PlotThin] table[ x expr=\thisrow{time}, y expr=\thisrow{tau_l}] {SimArtVeh/Trajectoryu.txt};
    \addlegendentry{$u_1$, mixed-exponents}

    \addplot[greenUni,PlotThin] table[ x expr=\thisrow{time}, y expr=\thisrow{tau_l}] {SimArtVeh/Trajectoryu_quad.txt};
    \addlegendentry{$u_1$, quadratic}

    \addplot[black,PlotOSavg, dashdotted] table[ x expr=\thisrow{time}, y expr=\thisrow{tau_gama}] {SimArtVeh/Trajectoryu.txt};
    \addlegendentry{$u_2$, mixed-exponents}

    \addplot[redUni,PlotOSavg, dashdotted] table[ x expr=\thisrow{time}, y expr=\thisrow{tau_gama}] {SimArtVeh/Trajectoryu_quad.txt};
    \addlegendentry{$u_2$, quadratic}
\end{axis}

\end{tikzpicture}%
    \caption{Closed-loop trajectory (left), optimal
    value function (middle), and input trajectories (right) for the articulated vehicle in a simulated parking maneuver. \enquote{Mixed-exponents} denotes the usage of stage costs~\eqref{eq:stage_cost_extended_robot} for the OCP~\eqref{eq:MPC_cont}, whereas for \enquote{quadratic} conventional quadratic stage costs were employed, cf.~Section~\ref{subsec:MPC Simulation Results}.}
    \label{fig:MPC}
\end{figure*}
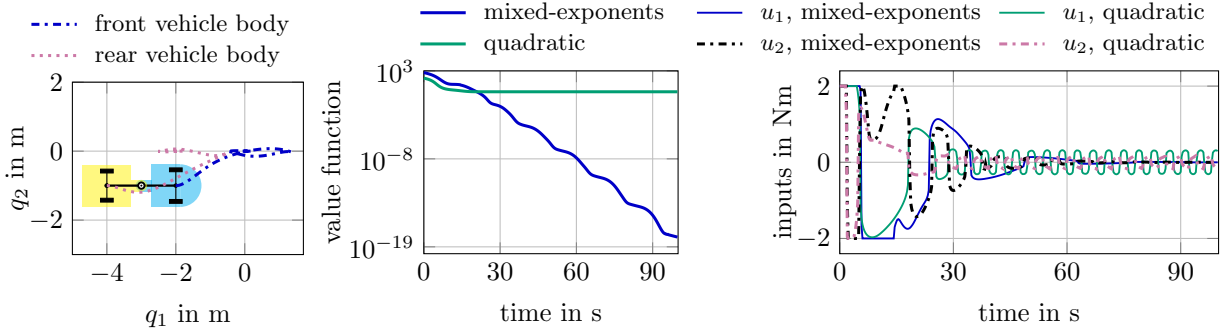
\subsection{MPC Simulation Results}
\label{subsec:MPC Simulation Results}
To illustrate the applicability of the proposed approach, a parking scenario for the articulated vehicle is considered in simulation.
For this purpose, the dynamic model is used, while corresponding results for a comparable kinematic robot-with-trailer system can be found in~\cite{RoseEbel22}.
A time-discretized version of the optimal control problem~\eqref{eq:MPC_cont} with stage costs \eqref{eq:stage_cost_extended_robot}, where $d=6$ and weighting coefficients $Q_{1}=Q_{2}=0.1$, $Q_{3}=Q_{4}=5$, $R_{1}=R_{2} = 0.125$, $\bar{R}_{1}=\bar{R}_{2}=1$, is formulated using CasADi~\cite{AnderssonEtAl19} and solved with Ipopt~\cite{WaechterBiegler06}, assuming a zero-order hold on the control input.
To this end, the system dynamics used within the MPC prediction model is discretized using the Runge–Kutta–Fehlberg 4(5) method with a fixed integration step size.
The sampling time is chosen as $\Delta t = \unit[0.25]{s}$ and the prediction horizon as $T = \unit[10]{s}$.
Moreover, input constraints are imposed as $\lvert u_j \rvert \leq \unit[2]{Nm}$, $j\in\mathbb{Z}_{1:2}$.
The vehicle is initialized at the position $q_1 = \unit[-2]{m}$, $q_2 = \unit[-1]{m}$, while all initial angles and velocities are set to zero.
The resulting closed-loop trajectory is shown in the plot on the left of Figure~\ref{fig:MPC}, where the vehicle converges to the setpoint $\p = \bm{0}$.
In particular, sub-millimeter and sub-millidegree accuracy is achieved for the first time at $t = \unit[52.25]{s}$, i.e., the position error falls below $\unit[10^{-3}]{m}$ in both the $q_1$ and $q_2$-directions, and the orientation error is below $\unit[10^{-3}]{\mathrm{deg}}$.

Motivated by previous observations that quadratic stage costs may fail to precisely control nonholonomic systems, cf.~\cite{muller_quadratic_2017,RoseEbel22,ebel_note_2023}, the middle plot of Figure~\ref{fig:MPC} compares the evolution of the value functions from MPC controllers using the proposed stage costs and conventional quadratic stage costs, respectively.
For the quadratic stage costs, the same weighting coefficients are used
where applicable.
While the value function 
associated with the proposed stage costs strictly decreases over time, the value function corresponding to the conventional quadratic stage costs decreases only marginally after approximately  $t=\unit[20]{s}$.
Around the same time, the corresponding control inputs obtained with the quadratic stage costs become small and subsequently alternate in sign, see the right plot of Figure~\ref{fig:MPC}. As a result, the vehicle remains nearly stationary and makes only negligible progress toward the setpoint, which is also reflected in the almost constant value function.

This indicates that, for the considered weighting coefficients, the quadratic stage costs do not provide a sufficient incentive to eliminate the remaining error.
Indeed, in the quadratic case, the vehicle does not reach the origin and the final pose reads
$\bm{q}(\unit[100]{s}) = \left[ \unit[0.37]{m} \ \unit[-0.04]{m} \ \unit[13.06]{deg} \ \unit[-33.70]{deg} \right]^\top$.

\section{Kinematic Systems: Homogeneous Approximation}
\label{sec:Homogeneous Approximation for Kinematic Systems}

In this section, we prove Proposition~\ref{prop:homogeneous_approx_kinematic} by constructing a 
homogeneous and locally null controllable approximation for the kinematic system.
In particular, we employ the nilpotent approximation introduced by Bella\"iche et al. as the “canonical nilpotent approximation" in~\cite{bellaiche_canonical_1993}. This approximation is constructed at a setpoint $\p\in\R^{n_q}$ of the driftless control-affine system~\eqref{eq:kinematics_nonholonomic_driftless} for which Assumption~\ref{ass:kin_sys_setpoint} holds, as we always presume in this section.
Then, the nilpotent approximation is obtained by transforming the system into privileged coordinates $\z(\q;\p)$ and truncating its Taylor expansion.

For this purpose, we first introduce the nonholonomic order of vector fields, which allows us to classify the terms in the Taylor series of the kinematics in privileged coordinates.
This classification is then used to construct the nilpotent approximation via truncation of the Taylor series. Following this, we prove that it is a homogeneous approximation after Definition~\ref{def:homogeneous_approximation}.
Some exemplary calculations for the considered example can be found in the appendix, see App.~\ref{app:example}.

\subsection{Nonholonomic Orders of the Control Vector Fields}
\label{subsec:Nonholonomic order for the control vector fields}

Based on the nonholonomic order for functions, see Table~\ref{tab:nonholsys}, we now define the nonholonomic order of vector fields. It characterizes the action of a vector field as a differential operator through the Lie derivative and, therefore, depends on the control vector fields of~\eqref{eq:kinematics_nonholonomic_driftless}. Since the nonholonomic order is a local attribute, we use the following notation: $C^{\infty}(\p)$ denotes the set of functions $f\!:\R^{n_q}\rightarrow\R$ that are smooth in a neighborhood of $\p$, and $C^{\infty,n_q}(\p)$ contains vector fields $\bm{g}\!:\R^{n_q}\rightarrow\R^{n_q}$ that are smooth in a neighborhood of $\p$.

\begin{definition}[Nonholonomic order of vector fields \cite{jean_control_2014}]
\label{def:nonhol_ord_vec_fields_and_operators}
    The nonholonomic order $\operatorname{ord}_{\p}(\bm{g})$ of a vector field $\bm{g}\in C^{\infty,n_q}(\p)$ at point $\p\in\R^{n_q}$ is defined as
    \begin{align}%
        \label{eq:def_nonhol_ord_vec_field}%
        \operatorname{ord}_{\p}(\bm{g})
        =
        \sup\! \left\{
        s\in\R\! :\! 
        \begin{array}{l}\!
        \operatorname{ord}_{\p}(\mathcal{L}_{\bm{g}}f)
        \geq s+\operatorname{ord}_{\p}(f),\!\\
        \forall f\in C^\infty(\p)
        \end{array}\!
        \right\}\!
        \end{align}
\end{definition}
\noindent and it follows that $\operatorname{ord}_{\p}(\bm{g})\in\mathbb{Z}$ since the nonholonomic orders of smooth functions are integers, except for the null vector field, which has infinite nonholonomic order, see~\cite{jean_control_2014}. 
\begin{lem}
\label{lem:nonh_ord_control_vec_fields}
Consider the driftless control-affine system with kinematics~\eqref{eq:kinematics_nonholonomic_driftless}, under Assumption~\ref{ass:kin_sys_setpoint}, with control vector fields $\bm{g}_j,\ j\in\Z_{1:n_v}$. 
It holds that each control vector field has a nonholonomic order of $-1$ in $\p$, i.e., $\operatorname{ord}_{\p}(\bm{g}_j)= -1,\ \forall\, j\in\Z_{1:n_v}$.
\end{lem}
\begin{proof}
The nonholonomic order of smooth functions is by definition an integer (cf.~\cite{jean_control_2014}). 
According to Definition~\ref{def:nonhol_ord_vec_fields_and_operators}, if a control vector field $\bm{g}_j$, $j\in\Z_{1:n_v}$, had a nonholonomic order smaller than $-1$, then there would exist a function $f\in C^{\infty}(\p)$ for which the relation \begin{align*}
        \operatorname{ord}_{\p}(\mathcal{L}_{\bm{g}_j}f)< -1+\operatorname{ord}_{\p}(f)
    \end{align*}
holds.
This would directly contradict the definition of nonholonomic order: Since the nonholonomic order of a function~$\operatorname{ord}_{\p}(f)$ indicates the smallest order of the nonholonomic derivatives for which the value at $\p$ is nonzero, and the nonholonomic order is defined recursively, a Lie derivative of a function along a control vector field cannot decrease the nonholonomic order by more than one.
Therefore, the order of the control vector fields must be greater than or equal to $-1$, i.e., $\operatorname{ord}_{\p}(\bm{g}_j)~\geq~-1,\ \forall j\in\mathbb{Z}_{1:n_v}$.

By definition (cf. Definition~\ref{def:nonhol_ord_vec_fields_and_operators}), the order of the control vector fields  is equal to $-1$, if there is a function $f\in C^{\infty}(\p)$ such that 
\begin{align}
\label{eq:control_vector_fields_-1}
    \operatorname{ord}_{\p}(\mathcal{L}_{\bm{g_j}}f)=\operatorname{ord}_{\p}(f)-1
\end{align}
holds. Since the control vector fields are non-zero at $\p\in\R^{n_q}$, for each $j\in\Z_{1:n_v}$ there is an $i\in\mathbb{Z}_{1:n_q}$ for which $g_{i,j}(\p)\neq 0$ holds.
Then, the function
\begin{align*}
    f_j(\x)&=q_i-p_i
\end{align*}
fulfills the condition~\eqref{eq:control_vector_fields_-1}, which concludes the proof.   
\end{proof}

\subsection{Construction of the Nilpotent Approximation}
\label{subsec:Construction of the Nilpotent Approximation}
Based on the nonholonomic orders introduced above, we now construct the nilpotent approximation of the kinematic system~\eqref{eq:kinematics_nonholonomic_driftless}.
The first step of the construction 
is the transformation of the system into 
so-called privileged coordinates, see also Sec.~\ref{subsec:Systematic Design Procedure for Kinematic Systems}.
We determine the privileged coordinates using \emph{Bella\"iche's algorithm} as presented in~\cite[p. 23]{jean_control_2014}, which has also been employed in~\cite{RoseEbel22}.

Under the coordinate transformation, 
the kinematic system~\eqref{eq:kinematics_nonholonomic_driftless} is transformed into the system
\begin{align}
\label{eq:system_privileged}
    \dotz &=\sum_{j=1}^{n_v}\bm{Z}_j(\z)v_j
\end{align}
with vector fields $\bm{Z}_j:\R^{n_q}\rightarrow\R^{n_q},\ j\in\Z_{1:n_v}$.

For the second step of the construction of the nilpotent approximation, each vector field $\bm{Z}_j$, $j\in\Z_{1:n_v}$, of system~\eqref{eq:system_privileged} is expanded as a Taylor series, whose entries consist of sums of monomials. 
These monomials are then sorted by their nonholonomic order. 
To this end, we first prove some preliminary statements, then define the necessary notation to write the kinematics~\eqref{eq:system_privileged} with explicitly sorted vector fields, and finally define the nilpotent approximation, where the vector fields are appropriately truncated.
To describe the monomials in privileged coordinates $\bm{z}=[z_1\ \cdots\ z_{n_q}]^\top\in\R^{n_q}$, we use a multi-index $\bm{\nu}=[\nu_1 \ \cdots \ \nu_{n_q}]^\top\in\N_{0}^{n_q}$ (cf.~\cite[p. 20]{jean_control_2014}) to represent the exponents in the form
\begin{align*}
    \bm{z}^{\bm{\nu}}&=\prod_{i=1}^{n_q} z_i^{\nu_i}.
\end{align*}  
Through combination of the multi-index with the nonholonomic weights, we obtain the weighted degree, which is then proven to be equal to the nonholonomic order of monomials.
\begin{definition}[Weighted degree, {\cite[p.~20]{jean_control_2014}}]
\label{def:weighted_degree}
    Let $\p\in\R^{n_q}$ denote the setpoint, $\bm{w}\!: \bm{p}\mapsto [w_1(\p) \ \cdots \ w_{n_q}(\p)]^\top$ the vector-valued function defining the nonholonomic weights of our system, and $\bm{\nu}=[\nu_1\ \cdots\ \nu_{n_q}]^{\top}\in\N_{0}^{n_q}$ a multi-index. Then, the weighted degree $\omega~:~\N_{0}^{n_q}~\times~\R^{n_q}~\rightarrow~\N_0$ at point $\p$ depends on the multi-index~$\bm{\nu}$ and is defined as 
    \begin{align*}
       \omega(\bm{\nu};\p)\coloneqq\bm{\nu}^{\top}\bm{w}(\p)=\sum_{i=1}^{n_q}\nu_iw_i(\p).
    \end{align*}    
\end{definition}
\begin{lem}
\label{lem:weighted_deg_eq_nonhol_ord}
    For a kinematic system~\eqref{eq:kinematics_nonholonomic_driftless} with setpoint $\p\in\R^{n_q}$, let Assumption~\ref{ass:kin_sys_setpoint} hold. With $\w(\p)\in\N^{n_q}$ denoting the nonholonomic weights of the system and $\z=\z(\q;\p)$ the privileged coordinates, the weighted degree $\omega(\bm{\nu};\p)=\sum_{i=1}^{n_q}w_i(\p)\nu_i$ of a monomial $\bm{z}^{\bm{\nu}}=\Pi_{i=1}^{n_q}z_i^{\nu_i}$ is equal to its nonholonomic order $\operatorname{ord}_{\p}(\z^{\bm{\nu}})$.
\end{lem}
\begin{proof}
    From Lemma~\ref{lem:operator_nonhol_order_extra}, see App.~\ref{app:Lemmata}, 
    the nonholonomic order of a product of two functions $f,g\in C^{\infty}(\p)$ satisfies
    \begin{align*}       
        \operatorname{ord}_{\p}(fg)&=\operatorname{ord}_{\p}(f)+\operatorname{ord}_{\p}(g).
    \end{align*}
    Since $\operatorname{ord}_{\p}(z_i)=w_i(\p)$ holds for $ i\in\Z_{1:n_q}$ by definition for privileged coordinates, the statement follows directly from the definition of the weighted degree, see Definition~\ref{def:weighted_degree}.
\end{proof}
Since the nonholonomic order of a sum of functions is bounded from below by the smallest nonholonomic order of its summands 
(cf. Lemma~\ref{lem:operator_nonhol_order_extra} in Appendix~\ref{app:Lemmata}), it follows that a Taylor series of an analytic function contains only monomials with a nonholonomic order greater than or equal to that of the function itself with at least one monomial having an equal nonholonomic order. To apply these results to the Taylor expansion of the control vector field~$\bm{Z}_j (\bm{z})$, we first determine the nonholonomic order of the functions $Z_{i,j}$ for $i\in\mathbb{Z}_{1:n_q}$.
\begin{lem}
\label{lem:ord_vec_fields_equal}
Consider the driftless control-affine system \eqref{eq:kinematics_nonholonomic_driftless}, with the setpoint $\p\in\R^{n_q}$, and let Assumption~\ref{ass:kin_sys_setpoint} hold. We express the kinematics in the privileged coordinates~\eqref{eq:system_privileged} for each $z_i=z_i(\q;\p),\ i\in\Z_{1:n_q}$, with $Z_{i,j}:\R^{n_q}\rightarrow\R,\ j\in\Z_{1:n_v},$ denoting the corresponding $i$-th row of the vector field~$\bm{Z}_j$.
Then, the nonholonomic order in $\p$ of the functions $Z_{i,j}$ with $Z_{i,j}\not\equiv 0$
is equal to $w_i(\p)-1$ for all $i\in\Z_{1:n_q},\ j\in\Z_{1:n_v}$.
\end{lem}
\begin{proof}
The vector fields $\bm{Z}_{j}$ have the same order as the control vector fields $\bm{g}_j$ (cf.~\cite{jean_control_2014,bellaiche_canonical_1993}), namely $-1$ for all $j\in\Z_{1:n_v}$ (cf. Lemma~\ref{lem:nonh_ord_control_vec_fields}).
Using the unit vectors $\partial_{z_i}$ for the coordinates $z_i,\ i\in\Z_{1:n_q}$, we write the vector fields of the kinematics in privileged coordinates~\eqref{eq:system_privileged}, for all $j\in\Z_{1:n_v}$, as
\begin{align*}
    \bm{Z}_{j}&=\sum_{i=1}^{n_q}Z_{i,j}(\z)\partial_{z_i}.
\end{align*}
Then, considering~\cite[Proposition 2.2]{jean_control_2014} as well as Bella\"iche et al. \cite{bellaiche_canonical_1993}, we conclude that for all $i\in\Z_{1:n_q},j\in\Z_{1:n_v}$, with $Z_{i,j}\not\equiv0$,
\begin{align*}
    \operatorname{ord}_{\p}(Z_{i,j})&= w_i-1
\end{align*}
holds.
\end{proof}

It now follows that the Taylor series of the vector fields $\bm{Z}_j,\ j\in\Z_{1:n_v},$ 
in privileged coordinates~\eqref{eq:system_privileged} consists of monomials whose weighted degrees~$\omega(\bm{\nu};\p)$ satisfy
\begin{align}
\label{eq:weighted_degree_geq-1}
    w(\bm{\nu};\p)-w_i(\p)&\geq -1.
\end{align}
Further, in each row there exists a monomial with a nonholonomic order for which the it holds that 
\begin{align}
\label{eq:weighted_degree_eq-1}
    \omega(\bm{\nu};\p)-w_i(\p)&= -1.
\end{align}

In the following, we abbreviate $\bm{w}=\bm{w}(\bm{p})$, as well as $\omega(\bm{\nu})=\omega(\bm{\nu};\p)$,  for simplicity.
We define a set containing all multi-indices with the same weighted degree offset by the weight $w_i$ which corresponds to the row in which the monomial is placed, as
\begin{align}
\label{eq:weighted_indices_set}
   I_{\nu}(\bm{w},i,k)&\coloneqq \big\{\bm{\nu}\in\N_{0}^{n_q}:\omega(\bm{\nu})-w_i=k \big\} 
\end{align} for $k\in\N_0\cup\{-1\}$. Hence, $I_{\nu}(\bm{w},i,-1)$ contains precisely all the multi-indices for which \eqref{eq:weighted_degree_eq-1} is fulfilled, and $I_{\nu}(\bm{w},i,k)$ for $k\in\N_0$ contains the monomials for which \eqref{eq:weighted_degree_geq-1} holds with a strict inequality (and therefore \eqref{eq:weighted_degree_eq-1} does not hold).

Using this notation, the Taylor series of the vector field $\bm{Z}_j,\ j\in\Z_{1:n_v}$, can be written row-wise as 
\begin{align*}
    Z_{i,j} (\bm{z})&= \sum_{k\in\N_0\cup\{-1\}}\,\sum_{\bm{\nu}\in I_\nu(\w,i,k)}c_{i,j,\bm{\nu}}\,\bm{z}^{\bm{\nu}},%
\end{align*}
with $i\in\Z_{1:n_q}$, $c_{i,j,\bm{\nu}}\in\R$ denoting the coefficient of the corresponding monomials.
To group monomials of equal weighted degree, we define functions containing only monomials with multi-indices $\bm{\nu}$ in the set $I_\nu(\w,i,k)$ for all $k\in\N_0\cup\{-1\}$, namely
\begin{align}\label{eq:sum_weightedDegree_l}
    Z^{[k]}_{i,j}(\bm{z})&\coloneqq \sum_{\bm{\nu}\in I_\nu(\w,i,k)}c_{i,j,\bm{\nu}}\,\bm{z}^{\bm{\nu}}.
\end{align}
Each monomial appears in exactly one function~$\bm{Z}_{j}^{[k]}$ so that $\sum_{k\in\N_{0}\cup\{-1\}}Z^{[k]}_{i,j}=Z_{i,j}$ holds.
Thus, we write the kinematics in privileged coordinates~\eqref{eq:system_privileged} as
\begin{align}
    \dot{\bm{z}}
    &=\sum_{j=1}^{n_v}\sum_{k\in\N_0\cup\{-1\}}\bm{Z}_{j}^{[k]}v_j.
\label{eq:system_privileged_taylor}
\end{align}

Finally, following \cite{bellaiche_canonical_1993}, the nilpotent approximation is obtained by the truncation of the Taylor expansion~\eqref{eq:system_privileged_taylor},
keeping only the vector fields~$\bm{Z}_{j}^{[-1]}$ that consist of monomials with a nonholonomic order of~$-1$. This yields an approximate system
\begin{align}
\label{eq:nilpotent_approximation_kinematics}
    \dot{\z}=\bar{\bm{f}}(\z,\bm{v}) &= \sum_{j=1}^{n_v} \bm{Z}_{j}^{[-1]}v_j,  
\end{align}
which provides a local approximation of the full kinematics~\eqref{eq:system_privileged}.

%
%

\subsection{Homogeneity of the Nilpotent Approximation}\label{subsec:HomogeneityNilpApprox}

\begin{proposition}\label{prop:homogeneity}
    Consider the driftless control-affine system~\eqref{eq:kinematics_nonholonomic_driftless}, which is locally controllable in the regular setpoint $\p\in\R^{n_q}$, with control vector fields that are smooth and analytic in a neighborhood of $\p$. 
    Let the nilpotent approximation~\eqref{eq:nilpotent_approximation_kinematics} in $\p$ be constructed as described in Sec.~\ref{subsec:Construction of the Nilpotent Approximation}
    with nonholonomic weights~$\w(\p) \in \mathbb{N}^{n_q}$. 
    This nilpotent approximation is $(\bm{r},\bm{s},\tau)$-homogeneous according to Definition~\ref{def:homogeneousSystem} in its coordinates $\z$ and $\bm{v}$, with homogeneity parameters~$\bm{r} = \w$, $\bm{s} = \bm{1}_{n_v}$, and $\tau=0$. 
\end{proposition}
\begin{proof}  
    In order to show homogeneity, we have that the 
    nilpotent approximation~\eqref{eq:nilpotent_approximation_kinematics} of the kinematics is given by $\dot{\bm{z}}=\bar{\bm{f}}(\bm{z},\bm{v})$, meaning that for all $i\in\Z_{1:n_q}$
\begin{align*}
    \bar{f}_i(\bm{z},\bm{u})=\sum_{j=1}^{n_v} Z_{i,j}^{[-1]}v_j
    =\sum_{j=1}^{n_v}\sum_{\bm{\nu}\in I_\nu(\w,i,-1)}c_{i,j,\bm{\nu}}\, \bm{z}^{\bm{\nu}}\, v_j
\end{align*}
holds.
For every monomial~$\bm{z}^{\bm{\nu}}$, the equality
\begin{align}
    (\Lambda_{\alpha}\z)^{\bm{\nu}}&=\prod_{i=1}^{n_q}(\alpha^{r_i}z_i)^{\nu_i}
    =\alpha^{(\nu_1r_1+\ldots+\nu_{n_{q}}r_{n_q})}\prod_{i=1}^{n_q}z_i^{\nu_i}  \label{eq:alpha_weighted_index}\nonumber \\
    &=\alpha^{\omega(\bm{\nu})}\z^{\bm{\nu}}. 
\end{align}
holds. 
Then, we set $\bm{s}=\bm{1}_{n_v}$ and $\tau=0$, and use the weighted degree~$\omega(\bm{\nu})$ (cf.~Definition~\ref{def:weighted_degree}),  in order to state that for all $i\in\Z_{1:n_q}$ it holds that
\begin{align*}                      \bar{f}_i(\bm{\Lambda}_{\alpha}\bm{z},\bm{\Delta_\alpha}\bm{v})&\overset{\hphantom{\eqref{eq:alpha_weighted_index}}}{=}\sum_{j=1}^{n_v}\sum_{\bm{\nu}\in I_\nu(\w,i,-1)}c_{i,j,\bm{\nu}}\,(\Lambda_{\alpha}\bm{z})^{\bm{\nu}}\,\alpha v_j\\
    &\overset{\eqref{eq:alpha_weighted_index}}{=}\sum_{j=1}^{n_v}\sum_{\bm{\nu}\in I_\nu(\w,i,-1)}\alpha^{(\omega(\bm{\nu})+1)}\, c_{i,j,\bm{\nu}}\,\bm{z}^{\bm{\nu}}\, v_j\\
    &\overset{\eqref{eq:weighted_indices_set}}{=} \alpha^{w_i}\underbrace{\sum_{j=1}^{n_v}\sum_{\bm{\nu}\in I_\nu(\w,i,-1)}c_{i,j,\bm{\nu}}\,\bm{z}^{\bm{\nu}}\, v_j}_{\displaystyle =\bar{f}_i(\bm{z},\bm{v})} \ .
\end{align*}
Therefore, each component $\bar{f}_i$, $i\in\Z_{1:n_q}$, fulfills the homogeneity condition~\eqref{eq:homogeneity_condition_definition} with $r_i\coloneqq w_i$.
Thus, the nilpotent approximation is $(\bm{r},\bm{s},\tau)$-homogeneous, and the weights $\bm{w}$ are equivalent to the homogeneity parameters $\bm{r}$. 
\end{proof}
It follows from the equality $\bm{r} = \bm{w}$ and Assumption~\ref{ass:kin_sys_setpoint}~(i) that the parameters $\bm{r}$ defining a $(\bm{r},\bm{1}_{n_v},0)$-homogeneity triplet are unique.
%
%
Now, we generalize the findings from Proposition~\ref{prop:homogeneity} in order to gain more insight into the properties of the residual functions, which result from the truncation of the system into a nilpotent approximation.
\begin{cor}
\label{cor:hom_all}
We consider the driftless control-affine system~\eqref{eq:kinematics_nonholonomic_driftless} with setpoint $\p\in\R^{n_q}$, and let Assumption~\ref{ass:kin_sys_setpoint} hold. The functions $Z_{i,j}^{[k]}:\R^{n_q}\rightarrow\R$, for $k\in\N_0\cup\{-1\}$, are obtained by sorting the vector fields $\bm{Z}_j$, $j\in\Z_{1:n_v}$, that describe the kinematics in privileged coordinates~\eqref{eq:system_privileged} (cf.~Section~\ref{subsec:Systematic Design Procedure for Kinematic Systems}).

Then, for every $k\in\N_0\cup\{-1\}$, the system with kinematics~$\bar{f}_{i}^{[k]}(\z,\bm{v})\coloneqq\sum_{j=1}^{n_v}Z_{i,j}^{[k]}(\z)\, v_j$
is homogeneous according to Definition~\ref{def:homogeneousSystem} in coordinates $\z$ and $\bm{v}$, with homogeneity parameters $r_i=w_i$ for $i\in\Z_{1:n_q}$, $\bm{s}=\bm{1}_{n_v}$, and $\tau =k+1$.
\end{cor}
\begin{proof}
The proof is similar to the one for Proposition~\ref{prop:homogeneity}. 
For the weighted degree,
\begin{align}
\label{eq:weighted_degree_minus_weight=k}
 \omega(\bm{\nu})-w_i=k   
\end{align}
holds for all rows $i\in\Z_{1:n_q}$. 
Further, by choosing $\bm{s}=\bm{1}_{n_v}$ and $\tau=k+1$, we obtain, 
\begin{align*}          
    \bar{f}_{i}^{[k]}(\bm{\Lambda}_{\alpha}\bm{z},\bm{\Delta_\alpha}\bm{v})&\overset{\hphantom{\eqref{eq:alpha_weighted_index}}}{=}\sum_{j=1}^{n_v}\sum_{\bm{\nu}\in I_\nu(\w,i,k)}c_{i,j,\bm{\nu}}\,(\Lambda_{\alpha}\bm{z})^{\bm{\nu}}\,\alpha v_j\\
    &\overset{\eqref{eq:alpha_weighted_index}}{=}\sum_{j=1}^{n_v}\sum_{\bm{\nu}\in I_\nu(\w,i,k)}\alpha^{(\omega(\bm{\nu})+1)}\, c_{i,j,\bm{\nu}}\,\bm{z}^{\bm{\nu}}\, v_j\\
    &\overset{\eqref{eq:weighted_degree_minus_weight=k}}{=} \alpha^{k+1}\alpha^{w_i}\underbrace{\sum_{j=1}^{n_v}\sum_{\bm{\nu}\in I_\nu(\w,i,k)}c_{i,j,\bm{\nu}}\,\bm{z}^{\bm{\nu}}\, v_j}_{\displaystyle =\bar{f}_{i}^{[k]}(\bm{z},\bm{v})}.
\end{align*}
For all $i\in\Z_{1:n_q}$, $\bar{f}_{i}^{[k]}$ fulfills the condition~\eqref{eq:homogeneity_condition_definition}, and therefore the corresponding system with kinematics~$\dot{\bm{z}}=\bar{f}_{i}^{[k]}(\z,\bm{v})$
is $(\bm{w},\bm{1}_{n_v},k+1)$-homogeneous. 
\end{proof}
Now, we consider the second prerequisite for the proof of Theorem~\ref{th:MPC}, which is that the nilpotent approximation~\eqref{eq:nilpotent_approximation_kinematics} is a homogeneous approximation of the kinematic system~\eqref{eq:system_privileged} in privileged coordinates near the origin according to Definition~\ref{def:homogeneous_approximation}.
\begin{proposition}
\label{prop:homoApproxNilpotent}
    We consider the driftless control-affine system~\eqref{eq:kinematics_nonholonomic_driftless} with setpoint $\p$ and let Assumption~\ref{ass:kin_sys_setpoint} hold. 
    Then there exist privileged coordinates $\z$ and a nilpotent approximation
    \begin{align*}
        \dotz &=\bar{\bm{f}}(\z,\bm{v})=\sum\nolimits_{j=1}^{n_v}\bm{Z}_{j}^{[-1]}v_j, \end{align*}
        which is $(\bm{w},\bm{1}_{n_v},0)$-homogeneous according to Proposition~\ref{prop:homogeneity}.
    Then, the nilpotent approximation is a $(\bm{w},\bm{1}_{n_v},0)$-homogeneous approximation of the system~\eqref{eq:system_privileged} in privileged coordinates near the origin according to Definition~\ref{def:homogeneous_approximation}, i.e., there exist $\rho,M,\eta>0$ such that~\eqref{eq:homogeneousApproximationResiduum} holds for the residual $\bm{R}(\z,\bm{v})\coloneqq\bm{f}(\z,\bm{v})-\bar{\bm{f}}(\z,\bm{v})$.
\end{proposition}
\begin{proof}
The residual $\bm{R}(\bm{z},\bm{v})$ of the nilpotent approximation is described by
\begin{align*}
\bm{R}(\bm{z},\bm{v}) &= \bm{f}(\bm{z},\bm{v}) - \bar{\bm{f}}(\bm{z},\bm{v})\\
&= \sum_{j=1}^{n_v}\sum_{k\in\N_0\cup\{-1\}}\bm{Z}_{j}^{[k]}(\bm{z}) v_j
- \sum_{j=1}^{n_v} \bm{Z}_{j}^{[-1]}(\bm{z}) v_j\\
&= \sum_{j=1}^{n_v}\sum_{k\in\N_0}\bm{Z}_{j}^{[k]} (\bm{z}) v_j.
\end{align*}
Since for all $k\in\N_0$ and $i\in\Z_{1:n_q}$ the functions $\bar{f}_{i}^{[k]}(\z,\bm{v})=\sum_{j=1}^{n_v}Z_{i,j}^{[k]}v_j$ are $(\bm{r},\bm{s},\widetilde{\tau})$-homogeneous with $\bm{r}=\bm{w},\ \bm{s}=\bm{1}_{n_v}$ and $\widetilde{\tau}=k+1$  (cf. Corollary~\ref{cor:hom_all}), the bound 
\begin{align*}
    \left\lvert R_{i}(\bm{\Lambda}_{\alpha}\bm{z},\bm{\Delta}_{\alpha}\bm{v}) \right\rvert
    &=\left\lvert \sum_{j=1}^{n_v}\sum_{k\in\N_0}Z_{i,j}^{[k]} (\bm{\Lambda}_{\alpha} \bm{z}) \, \alpha v_j \right\rvert \\
    &=\left\lvert \sum_{j=1}^{n_v}\sum_{k\in\N_0}\alpha^{k+1+w_i}Z_{i,j}^{[k]} (\bm{z}) \,  v_j \right\rvert \\
    & =\alpha^{w_i+1}\left\lvert \sum_{j=1}^{n_v}\sum_{k\in\N_0}\alpha^{k}Z_{i,j}^{[k]} (\bm{z}) \,  v_j\right\rvert
\end{align*}
holds for all $i\in\Z_{1:n_q}$. 
Since the function $Z_{i,j}=\sum_{k\in\N_0}Z_{i,j}^{[k]}$ is analytical, its Taylor series converges absolutely inside a convergence radius of $R>0$. 
We assume that $\rho\in(0,R)\cap(0,1)$ and $\eta=1$.
Further, $\alpha\in(0,1]$ holds, and $\lVert \bm{v} \rVert\leq \rho$ holds, which implies that $\lvert v_j\rvert\leq\rho<1$ holds for all $j\in\Z_{1:n_v}$.
Then, we employ the Weierstrass M-test, 
which implies that for all $i\in\Z_{1:n_q}$ and all $j\in\Z_{1:n_v}$ the series $\sum_{k\in\N_0}\alpha^{k}Z_{i,j}^{[k]} (\bm{z})v_j$ converges absolutely and uniformly with a convergence radius equal to or larger than $R$. Since this series consists of continuous functions, we can bound the convergence values with a constant $C_{i,j}>0$ on the compact set where $\lVert\z\rVert,\lVert \bm{v}\rVert \leq \rho$ holds, 
\begin{align*}
\left\lvert\sum_{k\in\N_0}\alpha^{k}Z_{i,j}^{[k]} (\bm{z}) \,  v_j\right\rvert&\leq C_{i,j}, \quad \forall\,\lVert\z\rVert,\lVert \bm{v}\rVert \leq \rho.
\end{align*}
We define
\begin{align*}
    \bar{M}_i&=n_v  \max_{j\in\Z_{1:n_v}}C_{i,j}.
\end{align*}
Then, the residual can be bounded separately for the cases $\alpha \in (0, 1)$ and $\alpha = 1$.
For the case $\alpha\in(0,1)$, 
    \begin{align*}
      \left\lvert R_{i}(\bm{\Lambda}_{\alpha}\bm{z},\bm{\Delta}_{\alpha}\bm{v}) \right\rvert
    &=\alpha^{w_i+1}\left\lvert \sum_{j=1}^{n_v}\sum_{k\in\N_0}\alpha^{k}Z_{i,j}^{[k]} (\bm{z}) \,  v_j\right\rvert \\
    &\leq \alpha^{w_i+1}\sum_{j=1}^{n_v}\underbrace{\left\lvert \sum_{k\in\N_0}\alpha^{k}Z_{i,j}^{[k]} (\bm{z}) \,  v_j\right\rvert}_{\leq\, \bar{M}_i/n_v}\\
    & \leq \alpha^{w_i+1}\bar{M}_i
    \end{align*}
holds. 

For the case $\alpha=1$, since the residuum is a continuous function, there exists $\tilde{M}_i>0$ such that
    \begin{align*}\left\lvert R_{i}(\bm{\Lambda}_{1}\bm{z},\bm{\Delta}_{1}\bm{v}) \right\rvert&=\left\lvert R_{i}(\bm{z},\bm{v}) \right\rvert\\
    &\leq\tilde{M}_i,\quad \forall \lVert \z \rVert,\lVert \bm{v} \rVert \leq\rho
    \end{align*}
holds.
By combining these two bounds, the condition~\eqref{eq:homogeneousApproximationResiduum} is fulfilled for all $\alpha\in(0,1]$ , with $\rho\in(0,1)\cap(0,R)$, $M=\max_{i\in\mathbb{Z}_{1:n_q}}\{\bar{M}_i,\tilde{M}_i\}$, and $\eta=1$. 
\end{proof}
\begin{proof}[Proof of Proposition~\ref{prop:homogeneous_approx_kinematic}]
According to Proposition~\ref{prop:homoApproxNilpotent}, the residual of the nilpotent approximation of a kinematic system is bounded on an open, bounded and convex set around the setpoint $\p$. Since the nilpotent approximation is locally null controllable to the origin (cf. ~\cite[Lemma 2.1]{jean_control_2014}), it fulfills the conditions stated in Proposition~\ref{prop:homogeneous_approx_kinematic}.
\end{proof}
%
%
\section{Dynamic Extension: Homogeneous Approximation}
\label{sec:Homogeneous Approximation for Kinematic Systems with Dynamic Extension}
In this section, we consider nonholonomic systems with a dynamic extension as presented in Section~\ref{subsec:Dynamic Extension of Nonholonomic Kinematics}, and prove that there exists a corresponding homogeneous approximation which fulfills the conditions of Proposition~\ref{prop:homogeneous_approx_w_dyn_ext}.
The system's kinematics is described by \eqref{eq:kinematics_nonholonomic_driftless} and the dynamic extension by~\eqref{eq:dynamic_part_dynamic_system}, yielding the overall system dynamics
\begin{align*}
     \begin{bmatrix}
            \dot{\bm{q}} \\   \bm{M}(\bm{q})\dot{\bm{v}}
        \end{bmatrix} =
        \begin{bmatrix}
            \bm{G} (\bm{q}) \\ \bm{0}_{n_v\times n_v} 
        \end{bmatrix}\bm{v} 
        +
        \begin{bmatrix}
            \bm{0}_{n_q\times n_u} \\\bm{B}(\bm{q})
        \end{bmatrix}
        \bm{u}.
\end{align*}

%
%
To obtain a nilpotent approximation of the system with dynamic extension~\eqref{eq:dynamic_extension_system}, we do not construct privileged coordinates for the full system with state~$\bm{x}=[\bm{q}^\top\ \bm{v}^\top]^\top\in\R^{n_x}$.
Instead, we approximate the kinematic and dynamic components separately and subsequently combine the resulting approximations.
As previously shown, we derive the nilpotent approximation for the kinematic system~\eqref{eq:kinematics_nonholonomic_driftless} by regarding the velocity $\bm{v}$ as the input. Then, for a setpoint $\p\in\R^{n_q}$, weights~$\w=\bm{w}(\p) =[w_{1}(\p)\ \cdots\ w_{n_q}(\p)]$, and privileged coordinates~$\bm{z} \coloneqq \bm{z}(\bm{q};\p)$, satisfying the equality $\operatorname{ord}_{\p}(z_i(\bm{q};\p))=w_i(\bm{p})$, we construct the nilpotent approximation (cf.~Section~\ref{sec:Homogeneous Approximation for Kinematic Systems}).
According to Proposition~\ref{prop:homoApproxNilpotent}, this nilpotent approximation is a $(\bm{w}_{\bm{q}}, \bm{1}_{n_v}, 0 )$-homogeneous control system according to Definition~\ref{def:homogeneousSystem} with kinematics
\begin{align}
\label{eq:canonical_z_q}
    \dot{\bm{z}} = \bar{\bm{f}}(\bm{z}, \bm{v}).
\end{align}
The dynamics~\eqref{eq:dynamic_part_dynamic_system}, where $\bm{M}\!: \mathbb{R}^{n_q} \to \mathbb{R}^{n_v \times n_v}$ is a symmetric and positive definite matrix-valued function, is described by
\begin{equation}\label{eq:dynamic_part_dynamic_system_rearranged}
\dot{\bm{v}} =
    \bm{M}^{-1}(\bm{q}) \bm{B}(\bm{q}) \bm{u} \eqqcolon \widetilde{\bm{B}}(\bm{q})\bm{u},
\end{equation}
where $\bm{B}(\bm{q})\in\mathbb{R}^{n_v \times n_u}$ and $n_v = n_u$.
Then, the nilpotent approximation of the extension~\eqref{eq:dynamic_part_dynamic_system_rearranged} is formed at the setpoint~$\bm{p}$ by linearizing~\eqref{eq:dynamic_part_dynamic_system_rearranged} in the privileged coordinates~$\bm{z}$ of the kinematic subsystem. This results in 
\begin{equation}
    \dot{\bm{v}} = 
    \widetilde{\bm{B}}(\bm{\phi}^{-1}(\bm{0})) \bm{u}
    \eqqcolon \bar{\bm{B}}\bm{u}, 
\end{equation}
where $\bm{\phi}^{-1}(\z)=\bm{q}$ describes the reverse coordinate transformation.

Using the coordinates~$\widetilde{\bm{z}}(\x;\p) = [\bm{v}^\top \  \bm{z}^\top]^\top \in\mathbb{R}^{n_x}$, $n_x = n_q + n_v$, we define what we call the composed approximation via its dynamics
\begin{align}\label{eq:approximation_dynamic_extension}
    \dot{\widetilde{\bm{z}}} = \begin{bmatrix}
        \dot{\bm{v}} \\ \dot{\bm{z}} 
    \end{bmatrix} = \begin{bmatrix}
        \bm{0}_{n_v} \\
        \bar{\bm{f}}(\bm{z},\bm{v}) 
    \end{bmatrix} + \begin{bmatrix}
        \bar{\bm{B}} \\
        \bm{0}_{n_q\times n_v} 
    \end{bmatrix} \bm{u}.
\end{align}

\begin{proposition}
\label{prop:homogeneity_approx_sys_ext}
    We consider a kinematic system~\eqref{eq:kinematics_nonholonomic_driftless} with setpoint $\p\in\R^{n_q}$ and dynamic extension~\eqref{eq:dynamic_part_dynamic_system}, and let Assumption~\ref{ass:kin_sys_setpoint} hold.  The composed approximation for the full system defined above with dynamics~\eqref{eq:approximation_dynamic_extension} is homogeneous according to Definition~\ref{def:homogeneousSystem} with homogeneity parameters~$\bm{r} =  [\bm{1}_{n_v}^\top\ \bm{w}^\top]^\top$, $\bm{s} = \bm{1}_{n_u}$, and $\tau=0$.
\end{proposition}
\begin{proof}
    Regarding the composed approximation
    \begin{align*}
        \begin{bmatrix}
        \dot{\bm{v}} \\ \dot{\bm{z}} 
    \end{bmatrix} = \begin{bmatrix}
        \bm{0}_{n_v} \\
        \bar{\bm{f}}(\bm{z},\bm{v}) 
    \end{bmatrix} + \begin{bmatrix}
        \bar{\bm{B}} \\
        \bm{0}_{n_q\times n_v} 
    \end{bmatrix} \bm{u}
    \end{align*}
    we write for the homogeneity parameters
    \begin{align}
    \label{eq:lambda}
        \bm{\Lambda}_{\alpha} =
        \begin{bmatrix}
            \alpha \bm{I}_{n_v} & \bm{0}_{n_v \times n_q} \\
             \bm{0}_{n_q \times n_v} & \tilde{\bm{\Lambda}}_{\alpha}\\
        \end{bmatrix}\in\mathbb{R}^{n_x \times n_x}
    \end{align}
    with $\alpha\in(0,1]$ and $\tilde{\bm{\Lambda}}_{\alpha} = \operatorname{diag}\{ \alpha^{r_{n_v+1}}, \dots, \alpha^{r_{n_x}}\} \in\mathbb{R}^{n_q \times n_q} $.
    Further, for $\bm{s}=\bm{1}_{n_v}$ it follows that $\bm{\Delta}_{\alpha} = \alpha \bm{I}_{n_u}=\alpha \bm{I}_{n_v}$, and we choose $\tau=0$.
    Now, we look at the dynamics of $\bm{v}$ and $\bm{z}$ separately and check the homogeneity condition~\eqref{eq:homogeneity_condition_definition}.\\   
    \emph{First case:} For rows $i\in\Z_{1:n_v}$ we regard the dynamics for all $\widetilde{z}_i$ with $ i\in\mathbb{Z}_{1:n_v}$ and check the homogeneity condition~\eqref{eq:homogeneity_condition_definition}. For $i\in\mathbb{Z}_{1:n_v}$, since the dynamics is linear in $\bm{u}$, we get
        \begin{align*}       \dot{\widetilde{z}}_i(\bm{\Lambda}_{\alpha}\widetilde{\bm{z}},\bm{\Delta}_{\alpha}\bm{u})=\bar{\bm{B}}^\top_i\,\bm{\Delta}_{\alpha}\bm{u}
            =\alpha \bar{\bm{B}}^\top_i\bm{u},
        \end{align*}
    with $\bar{\bm{B}}^\top_i$ denoting the $i$-th row of $\bar{\bm{B}}$. \\    
    \emph{Second case:} For rows $i\in\mathbb{Z}_{n_v+1:n_x}$ the dynamics are described by the nilpotent approximation~\eqref{eq:canonical_z_q}, i.e., the function $h_{i-n_v}(\bm{z},\bm{v})$. 
    We regard the velocity $\bm{v}$ as the input of the kinematic subsystem and its nilpotent approximation, for which it was assumed that $\bm{s}=\bm{1}_{n_v}$, cf.~Prop.~\ref{prop:homogeneity}. Since the first $n_v$ entries of the homogeneity parameters $\bm{r}$ of the composed approximation as a whole are equal to $1$, we can proceed analogously, with $\bm{\Delta}_{\alpha}=\alpha \bm{I}_{n_v}$ equal to the first $n_v$ rows and columns of $\bm{\Lambda}_{\alpha}$, cf.~\eqref{eq:lambda}. 
    Then, it follows that 
        \begin{align*}         \dot{\widetilde{z}}_{i}(\bm{\Lambda}_{\alpha}\widetilde{\bm{z}},\bm{\Delta}_{\alpha}\bm{u})
        &=\bar{\bm{f}}_{i-n_v}(\tilde{\bm{\Lambda}}_{\alpha}\z,\bm{\Delta}_{\alpha}\bm{v})\\
        &=\alpha^{w_{i-n_v}}\bar{\bm{f}}_{i-n_v}(\z)
        \end{align*}
    holds.
    
    Therefore, for all $i\in\Z_{1:n_x}$, the system is $(\bm{r},\bm{1}_{n_v},0)$-homogeneous, with \begin{align*}
        r_i&=\begin{cases}
            1,&i\in\Z_{1:n_v}\\
            w_{i-n_v},&i\in\Z_{n_v+1:n_x}
        \end{cases}.
    \end{align*}
    Thus, the composed approximation is $(\bm{r},\bm{1}_{n_u},0)$-homogeneous, according to Definition~\ref{def:homogeneousSystem}, and $\widetilde{\bm{w}}~=~[\bm{1}_{n_v}^\top\ \bm{w}^\top]^\top$ is equivalent to the parameters~$\bm{r}$.
\end{proof}

\begin{proposition}\label{prop:extendedSystemApprox_homApprox}
 We consider a kinematic system~\eqref{eq:kinematics_nonholonomic_driftless} with setpoint $\p\in\R^{n_q}$ and dynamic extension~\eqref{eq:dynamic_part_dynamic_system}, and let Assumption~\ref{ass:kin_sys_setpoint} hold.  Then, there exists a composed approximation for the full system, and its dynamics~\eqref{eq:approximation_dynamic_extension} form a $(\bm{r},\bm{1}_{n_v},0)$-homogeneous approximation of the system with dynamic extension~\eqref{eq:dynamic_extension_system} in coordinates $\widetilde{\z}=[\bm{v}^{\top}\,\z^{\top}]^{\top}$ in its origin, according to Definition~\ref{def:homogeneous_approximation}, with homogeneity parameters $\bm{r} = [\bm{1}_{n_v}^\top,\bm{w}^\top]^\top$, with the residuum bound by the parameters  $\eta=1$ and $\rho,M>0$. 
\end{proposition}
\begin{proof}
As in the proof for Proposition~\ref{prop:homogeneity_approx_sys_ext}, we regard the approximations of $\bm{z}$ and $\bm{v}$ as well as their respective residual separately.\\
\emph{First case:} For rows $i\in\Z_{1:n_v}$ of the composed approximation 
%
the residual is defined as $ R_{i}(\widetilde{\z},\bu)=R^{v}_{i}(\widetilde{\z},\bu)$, where
\begin{align*}
    \bm{R}^{v}(\widetilde{\z},\bu)&=\widetilde{\bm{B}}(\bm{\phi}^{-1}(\bm{z}))\bu -\bar{\bm{B}}\bu.
\end{align*}
We assume that $\eta=1$ and $\rho\in(0,1)$. 
The residual consists of the non-constant terms of the Taylor series of $\widetilde{\bm{B}}(\bm{\phi}^{-1}(\bm{z}))$, which we write as
\begin{align*}
    R_{i}(\widetilde{\z},\bu)&=R_{i}(\z,\bu)=\sum_{j=1}^{n_v}\sum_{k=0}^{\infty}c_{i,j,k}\z^{\bm{\nu}_{i,j,k}}u_j,
\end{align*}
for $i\in\Z_{1:n_v}$, with $\bm{\nu}_{i,j,k}\in\N_{0}^{n_q}$ denoting the multi-index and $c_{i,j,k}~\in~\R$ the coefficients of the Taylor series. 

We define $\tilde{\Lambda}_{\alpha}\coloneqq\diag(\alpha^{w_1},\ldots,\alpha^{w_{n_v}})$.
Then 
\begin{align*}
   \lvert R_{i,j}(\Lambda_{\alpha}\widetilde{\z},\Delta_{\alpha}\bu)\rvert&=\left\lvert \sum_{k=0}^{\infty}c_{i,j,k}(\tilde{\Lambda}_{\alpha}\z)^{\bm{\nu}_{i,j,k}}\alpha u_j\right\rvert\\
   &= \alpha \left\lvert \sum_{k=0}^{\infty}c_{i,j,k}\alpha^{(\bm{\nu}_{k}^{\top}\bm{w})}\z^{\bm{\nu}_{i,j,k}}u_j\right\rvert\\
   &\leq \alpha^2\underbrace{\left\lvert \sum_{k=0}^{\infty}c_{i,j,k}\z^{\bm{\nu}_{i,j,k}}u_j\right\rvert}_{\leq C_{i,j}>0}
\end{align*}
holds, with the constant $C_{i,j}>0$, since $\bm{\nu}_{k}^{\top}\bm{w} \geq 1$ holds, $\widetilde{\bm{B}}(\bm{\phi}^{-1}(\cdot))$ is analytic, and we are on a compact set defined by $\lVert \widetilde{\z} \rVert,\lVert \bu\rVert\leq\rho$. 
We define
\begin{align*}
    M_v\coloneqq\max_{i,j\in\mathbb{Z}_{1:n_v}}C_{i,j}.
\end{align*}
Then, for $i\in\Z_{1:n_v}$, 
\begin{align*}
    \lvert R_i(\Lambda_\alpha \widetilde{\z},\Delta_\alpha\bu) \rvert &\leq \alpha^{2}n_vM_v
\end{align*}
holds.

\emph{Second case:} For rows $i\in\Z_{n_v+1:n_x}$ of the composed approximation the residual is equivalent to the residual of the nilpotent approximation of the kinematic subsystem. Therefore, we can use the same bound. We follow from Proposition~\ref{prop:homoApproxNilpotent} that for $i\in\Z_{n_v+1:n_x}$
\begin{align*}
    \lvert R_i(\Lambda_{\alpha}\widetilde{\z},\Delta_{\alpha}\bu)\rvert&\leq \alpha^{r_{i}+1}C_i,\quad \forall\lVert\widetilde{\z}\rVert,\lVert\bu\rVert \leq\rho
\end{align*} 
holds, with $\eta=1$, $C_i>0$, and $\rho\in(0,1)\cap(0,R)$, where $R>0$ denotes the convergence radius of the Taylor series of $h(\z)$. We define
\begin{align*}
    M_{z}\coloneqq\max_{i\in \mathbb{Z}_{n_v+1:n_x}}C_i.
\end{align*}
Then, the residuum for the composed approximation is bounded, with $M\coloneqq\max\{M_v,M_{z}\}$, $\rho\in(0,1)\cap(0,R)$, and $\eta=1$, which concludes the proof.
\end{proof}
\begin{proof}[Proof of Proposition~\ref{prop:homogeneous_approx_w_dyn_ext}]
According to Proposition~\ref{prop:extendedSystemApprox_homApprox}, the residuum of the composed approximation of a kinematic system with dynamic extension is bounded on an open, bounded and convex set around the setpoint~$\widetilde{\bm{p}}$. 
Further, since the nilpotent approximation $\dot{\z}=\bar{\bm{f}}(\z,\bm{v})$ (globally) fulfills the LARC~\cite{jean_control_2014}, the composed approximation~\eqref{eq:approximation_dynamic_extension} satisfies it as well because the additional input dynamics are fully actuated. 
This directly implies null controllability.
Therefore,  it is proven that the composed approximation fulfills the conditions of Proposition~\ref{prop:homogeneous_approx_w_dyn_ext}.    
\end{proof}

\section{Conclusions}
\label{sec:Conclusion}
The present paper establishes a systematic procedure for designing MPC controllers without terminal ingredients that locally asymptotically stabilize equilibria of nonholonomic systems.
The key design of the tailored stage costs follows directly from intrinsic mechanical properties of the system, namely the nonholonomic weights associated with the input distributions and the corresponding privileged coordinates, which can be obtained, for example, via Bella\"iche's algorithm.
The main contribution is the proof that the canonical nilpotent approximation is homogeneous with homogeneity parameters given by the nonholonomic weights, together with a bound on the associated approximation error.
These results eliminate the need for previously required lengthy technical computations and yield a straightforward design procedure that is readily applicable in practice.
Further, the mixed-exponents costs resulting from this design procedure have a natural mechanical interpretation as its state-dependent part is locally closely related to the sub-Riemannian distance, which captures the distance between two configurations of a nonholonomic system.
Beyond the conventional kinematic setting, the framework is generalized to nonholonomic kinematics with dynamic extensions.
Local asymptotic stability of the loop closed under the MPC feedback law is established for this system class, thereby broadening the applicability of the framework to a wider class of dynamical systems subject to nonholonomic constraints.
The effectiveness of the proposed controller is demonstrated throughout the paper using an articulated vehicle as an example.
In this way, this work answers several open questions raised in previous studies, see, e.g.,~\cite{RoseEbel22,ebel_note_2023}, and further strengthens the connection between sub-Riemannian geometry and homogeneous control systems in the context of MPC.
\appendix
\section{Appendix}
\subsection{Bella\"iche's Algorithm}

Bella\"iche's algorithm for deriving a set of privileged coordinates for the nonholonomic kinematic system~\eqref{eq:kinematics_nonholonomic_driftless}, provided that Ass.~\ref{ass:kin_sys_setpoint} holds, is presented in Algorithm~\ref{alg:AlgorithmBellaiche}.
For further details, we refer the interested reader to~\cite{jean_control_2014,bellaiche1996tangent}.

\label{subsec:BA}
\begin{algorithm}[!htbp]
\renewcommand{\algorithmicdo}{}
\caption{Bella\"{i}che's algorithm} 
\label{alg:AlgorithmBellaiche}

\begin{algorithmic}[1]
\small
\STATE{\bf input:} \begin{alglist}
        \item nonholonomic kinematics \begin{align*}\dot{\bm{q}} = \sum_{j=1}^{n_v} \bm{g}_j(\bm{q}) v_j \eqqcolon \bm{G}(\bm{q}) \bm{v}\end{align*}
        \item setpoint $\bm{p}=\bm{0}$, if necessary, shift the coordinates (which is always possible, since the system is driftless)
        \item the associated distributions and nonholonomic weights $\bm{w}=\bm{w}(\bm{p})$    
        \end{alglist}
\STATE choose a so-called adapted frame of vector fields $\bm{X}_i$, $i \in \Z_{1:n_q}$,  where
\begin{alglist}
\item $\bm{X}_i\in\Delta^{w_i}$, $i \in \Z_{1:n_q}$,
\item $\operatorname{span} \{\bm{X}_1(\bm{p}), \ldots, \bm{X}_{n_q}(\bm{p})\} = \Delta^r(\bm{p})$
\end{alglist}
\STATE compute coordinates $\bm{y}(\bm{q};\bm{p})=\bm{A}^{-1} \bm{q}$, where \\$\bm{A}\coloneqq \left. \left[\bm{X}_1(\bm{q})\ \cdots\ \bm{X}_{n_q}(\bm{q})\right]\right|_{\bm{q}=\bm{0}}$
\STATE compute the vector fields $\bm{Y}_i(\bm{y}) = (\bm{A}^{-1} \bm{X}_i\bm{A})\bm{y}$,\\
$i \in \Z_{1:n_q}$
\FOR{\(j=1,\dots,n_q\)} \label{algStep:forLoopStart}
    \FOR{\(k=2,\dots,w_j-1\)}
        \STATE compute $h_{j,k}(y_1,\dots,y_{j-1})$ as follows
            \begin{alglist}
                \item[\hspace{2.1em}I.\phantom{II}\hspace{-2em}]
                determine the multi-index set
                \begin{align*}\widetilde{I}_{\nu}(j,k)\coloneqq&\begin{Bmatrix}\bm{\nu}\in\N_{0}^{n_q}:\bm{\nu}^\top\bm{w}-w_j<0,\\
                \nu_i=0\  \forall\,i\in\Z_{j+1:n_q}, \sum_{i=1}^{j-1} \nu_i=k \end{Bmatrix}\end{align*}
                \item[\hspace{2.1em}II.\phantom{I}\hspace{-2em}]
                for each $\bm{\nu} \in \widetilde{I}_{\nu}(j,k)$, compute\\
                $c_{\bm{\nu}} = \left. \mathcal{L}_{\bm{Y}_1}^{\nu_1}\left(\dots\left(\mathcal{L}_{\bm{Y}_{j-1}}^{\nu_{j-1}}\left( y_j -\sum_{i=2}^{k-1}h_{j,i}\right)\right)\right)\right|_{\bm{y}=\bm{0}}$
                \item[\hspace{2.1em}III.\hspace{-2em}]
                compute
                \begin{align*}h_{j,k}(y_1,\dots,y_{j-1}) = \sum_{\bm{\nu} \in \widetilde{I}_{\nu}(j,k)}
                c_{\bm{\nu}}
                \bm{y}^{\bm{\nu}} \left(\prod_{i=1}^{j-1}\nu_i!\right)^{-1}\end{align*}
            \end{alglist}
    \ENDFOR \label{algStep:forLoopEnd}
    \STATE $\zeta_j(\bm{y};\bm{p})=y_j-\sum_{k=2}^{w_j-1}h_{j,k}(y_1,\dots,y_{j-1})$
    \STATE $z_j(\bm{q};\bm{p})=\zeta_j(\bm{A}^{-1} \bm{q};\bm{p})$
\ENDFOR
\STATE {\bf output:} a system of privileged coordinates~$\bm{z}$
\end{algorithmic}
\end{algorithm}

\subsection{Auxiliary Results on Nonholonomic Orders}
\label{app:Lemmata}
\begin{lem}[\!\!\cite{jean_control_2014}, p.\ 16]
\label{lem:operator_nonhol_order_extra}
Let $\p\in\R^{n}$ be a point in the state space.
\begin{enumerate}[(i)]
\item   For two functions $f_1,f_2\in C^{\infty}(\p)$ the following statements hold:
    \begin{align}
        \operatorname{ord}_{\p}(f_1+f_2)&\geq\min\{\operatorname{ord}_{\p}(f_1),\operatorname{ord}_{\p}(f_2)\}\\
    \label{eq:equality_nonh_ord_mult}
        \operatorname{ord}_{\p}(f_1\cdot f_2)&=\operatorname{ord}_{\p}(f_1)+\operatorname{ord}_{\p}(f_2).
    \end{align}   
\item The nonholonomic order cannot be changed through scalar multiplication, meaning that for all functions $f\in C^{\infty}(\p)$ and all vector fields $\bm{g}\in C^{\infty,n}(\p)$, with scalar $\alpha\in\R\setminus\{0\}$,
\begin{align*}
   \operatorname{ord}_{p}(\alpha f)&=\operatorname{ord}_{p} (f),\\
   \operatorname{ord}_{p} (\alpha \bm{g}) &= \operatorname{ord}_{p} (\bm{g}),
\end{align*}
holds.

\item For vector fields $\bm{g}_1,\bm{g}_2\in C^{\infty,n}(\p)$ it holds that
\begin{align}  \label{eq:nonh_ord_bracket}
\operatorname{ord}_{\p}(\bm{g}_1+\bm{g}_2)&\geq \min\{\operatorname{ord}_{\p}(\bm{g}_1),\operatorname{ord}_{\p}(\bm{g}_2)\}.
\end{align}
    \end{enumerate}
\end{lem}
\begin{proof}
The statements follow directly from the definition of the nonholonomic order for functions and vector fields (cf. Sec.~\ref{subsec:Systematic Design Procedure for Kinematic Systems}, Definition~\ref{def:nonhol_ord_vec_fields_and_operators}).
\begin{enumerate}[(i)]
\item Since there is a product rule for Lie derivatives, i.e., \begin{align*}
 \mathcal{L}_{\bm{g}}(f_1\cdot f_2)=f_1\cdot\mathcal{L}_{\bm{g}}(f_2)+f_2\cdot\mathcal{L}_{\bm{g}}(f_1),   
\end{align*} and a nonzero term only occurs when both functions have been derived at least as often as their nonholonomic order denotes, the equality~\eqref{eq:equality_nonh_ord_mult} follows directly from the definition of the nonholonomic order of functions (cf. Sec.~\ref{subsec:Systematic Design Procedure for Kinematic Systems}).
\item Since the nonholonomic order of a function is defined by wether its nonholonomic derivatives vanish at $\p$, and a scalar factor is not impacted by Lie derivation, the scalar factor cannot influence its nonholonomic order. The same holds for vector fields.
\item By definition, we get \begin{align*}
       &\operatorname{ord}_{\p}(\bm{g}_1+\bm{g}_2)\\   &=\sup\begin{Bmatrix}s\in\R:&\operatorname{ord}_{\p}(\mathcal{L}_{\bm{g}_1+\bm{g}_2}f)\geq s+\operatorname{ord}_{\p}(f)\\ & \forall f\in C^\infty (\p)\end{Bmatrix}.
   \end{align*}
   For the Lie derivative,
\begin{align*}\mathcal{L}_{\bm{g}_1+\bm{g}_2}f&=\mathcal{L}_{\bm{g}_1}f+\mathcal{L}_{\bm{g}_2}f
   \end{align*}
   holds \cite{coron_control_2009}. Then a nonholonomic derivative of $\mathcal{L}_{\bm{g}_1}f+\mathcal{L}_{\bm{g}_2}f$ is nonzero in $\p$ only when the same holds for either $\mathcal{L}_{\bm{g}_1}f$ or $\mathcal{L}_{\bm{g}_2}f$, and therefore \eqref{eq:nonh_ord_bracket} holds.
\end{enumerate}%
\end{proof}
\subsection{Local Equivalence of the State-Dependent Part of the MPC Stage Costs to the Sub-Riemannian Distance}
\label{app:equivalence_subRiem_cost}
Following~\cite{jean_control_2014}, given a set of privileged coordinates $\bm{z}(\bm{q};\bm{p})$, there exist constants $C_{\p}, \varepsilon_{\p}>0$ such that for $d_{\textnormal{R}}(\bm{q},\p) < \varepsilon_{\p}$,
\begin{align*}
    \frac{1}{C_{\p}} \sum_{i=1}^{n_q}\lvert z_i(\bm{q},\p) \rvert^{\frac{1}{w_i}}
    \leq d_{\textnormal{R}}(\bm{q},\p)
    \leq C_{\p} \sum_{i=1}^{n_q}\lvert z_i(\bm{q},\p) \rvert^{\frac{1}{w_i}},
\end{align*}
where $d_{\textnormal{R}}(\bm{q},\p)$ denotes the sub-Riemannian distance between $\bm{q}$ and $\p$. 
Using the equivalence between $r$-homogeneous $p$-norms as stated in~\cite{zhang2022homogeneous},
it follows that, for $d \geq 1$,
\begin{align*}
\sum_{i=1}^{n_q} |z_i(\bm{q},\p)|^{\frac{1}{w_i}}
\sim
\left(\sum_{i=1}^{n_q} |z_i(\bm{q},\p)|^{\frac{d}{w_i}}\right)^{1/d}.
\end{align*}
Consequently, the state-dependent part of the stage costs is locally equivalent to $d_{R}(\bm{q},\bm{p})^d$, and thus, up to taking the $d$-th root, to the sub-Riemannian distance itself, i.e.,
\begin{align*}
\left( \sum_{i=1}^{n_q}\lvert z_i(\bm{q},\p) \rvert^{\frac{d}{w_i}} \right) ^{\frac{1}{d}} \sim d_{\textnormal{R}}(\bm{q},\p).
\end{align*}
The same local equivalence can also be inferred from~\cite{bellaiche1996tangent}.

\subsection{Remarks on the Nilpotent Approximation and Examples}\label{app:example}
%
\begin{rem}
    For a two-input nilpotent approximation in chained form given by
    \begin{align*}
        \dot{z}_i
        &=
        \begin{cases}
            v_i, & i\in\Z_{1:2},\\
            z_{i-1}v_1, & i\in\Z_{3:n_q},
        \end{cases}
    \end{align*}
    the homogeneity
    condition can be verified directly. With $w_1=w_2=1$, for $\bm{r}=\bm{w}$, $\bm{s}=\bm{1}_2$, and $\tau=0$,
    the components
    $i\in\Z_{1:2}$ satisfy the condition because $s_i=r_i=w_i=1$. For
    $i\in\Z_{3:n_q}$, the homogeneity condition reads
    \begin{align*}
        \bar{f}_i(\Lambda_\alpha\bm{z},\Delta_\alpha\bm{v})
        =
        \alpha^{w_{i-1}+s_1}z_{i-1}v_1
        =
        \alpha^{w_i}z_{i-1}v_1,
    \end{align*}
    since $s_1=1$ and $w_i=w_{i-1}+1$.
    Thus, the chained-form approximation is
    \((\bm w,\bm{1}_2,0)\)-homogeneous.
\end{rem}

\begin{rem}
For the nilpotent approximation~\eqref{eq:nilpotent_approximation_kinematics}, since $w_i=1$ holds for all $i\in\Z_{1:n_v}$, and Equation~\eqref{eq:weighted_degree_geq-1} must be fulfilled for all monomials in the functions $Z^{[-1]}_{i,j}$, it follows that the corresponding  multi-indices~$\bm{\nu}$ must all be equal to $\bm{0}_{n_q}$ for the first $n_v$ rows.
Furthermore, if we assume that for a monomial in $Z_{i,j}^{[-1]}$, $i\in\Z_{1:n_q}$, that there is a $j\geq i$ for which $\nu_j >0$ holds, it follows that $w(\bm{\nu},\p)\geq w_i$ holds, which contradicts \eqref{eq:weighted_degree_geq-1}. Therefore, the functions $Z_{i,j}^{[-1]}$ are only dependent upon the variables $z_1,\ldots,z_{i-1}$ for all $i\in\Z_{1:n_q}$.
Consequently, the nilpotent approximation is of triangular shape,
    \begin{align}\label{eq:triangularShape}
        \dot{z}_i &= \begin{cases} v_i,  & i\in\Z_{1:n_v},\\
        \sum_{j=1}^{n_v}Z_{i,j}^{[-1]}(z_1,\ldots,z_{i-1}) \,v_j, & i\in\Z_{n_v+1:n_q}.\end{cases}
    \end{align}  
See also \cite[p. 25]{jean_control_2014} and~\cite{bellaiche_canonical_1993} for equivalent statements.
\end{rem}
\begin{rem}
    The triangular structure of the nilpotent approximation~\eqref{eq:triangularShape} is related to chained systems~\cite{MurrSast93}, where the first $n_v$ states form an integrator chain as well. 
    However, triangularity by itself does not imply the existence of a chained-form representation~\cite{MurrSast93}.
    The present analysis does not require a transformation into chained form, since the stability result is based on homogeneity  of the nilpotent approximation.
    Conversely, chained-form systems possess the required triangular structure and therefore fall within the class of systems covered by the proposed MPC approach. For a different approach to the control of chained-form nonholonomic systems, see~\cite{Defoort16}.
\end{rem}
\begin{rem}
The nilpotent approximation is called the canonical nilpotent approximation in~\cite{bellaiche_canonical_1993},\cite{jean_control_2014}, because it has an inherent structure and it is unique in regard to the privileged coordinates. The privileged coordinates may not be unique to a system, as multiple sets of privileged coordinates with respective canonical nilpotent approximations may exist, compare~\cite{bellaiche_canonical_1993} or~\cite{SklyarIgnatovich07}. 
\end{rem}
\begin{example}
    To obtain the nilpotent approximation of the articulated vehicle, first, the transformed kinematic system~\eqref{equ:kin_artVeh_z} is expanded into its Taylor series, 
    and the weighted degrees of the resulting monomials are computed.    
    This yields
    \begin{align*}
    \label{equ:trailer_artVeh}
    \dot{\bm{z}} &= 
    \begin{bmatrix} 
        \bm{z}^{\bm{0}} - \frac{1}{8}\bm{z}^{[0 \, 2 \, 0 \, 0]} + \cdots \\   
        0 \\
        - \bm{z}^{[0 \, 1 \, 0 \, 0]} - \frac{1}{12}\bm{z}^{[0 \, 3 \, 0 \, 0]} + \cdots \\
        - \bm{z}^{[0 \, 0 \, 1 \, 0]} - \frac{1}{8}\bm{z}^{[0 \, 3 \, 0 \, 0]} + \cdots 
    \end{bmatrix}
    v_{1} + 
    \begin{bmatrix} 
        0 \\   
        \bm{z}^{\bm{0}} \\
        - \frac{1}{4} \bm{z}^{[0 \, 2 \, 0 \, 0]}\\
        - \frac{1}{4} \bm{z}^{[0 \, 2 \, 0 \, 0]}
    \end{bmatrix}
    v_{2}
    \nonumber\\
    &= 
    \begin{bmatrix} 
        Z_{1,1}^{[-1]} + Z_{1,1}^{[1]}+ \cdots \\   
        0 \\
        Z_{3,1}^{[-1]} + Z_{3,1}^{[1]}+ \cdots \\
        Z_{4,1}^{[-1]} + Z_{4,1}^{[0]}+ \cdots
    \end{bmatrix}
    v_{1} + 
    \begin{bmatrix} 
        0 \\   
        Z_{2,2}^{[-1]} \\
        Z_{3,2}^{[0]} \\
        Z_{4,2}^{[-1]}
    \end{bmatrix}v_{2}.
    \end{align*}
 %
 Truncating the Taylor series by keeping only the monomials $~Z_{i,j}^{[-1]}$, $i \in \mathbb{Z}_{1:4}$, $j \in \mathbb{Z}_{1:2}$, that have a nonholonomic order of~$-1$ leads to the nilpotent approximation
    \begin{equation}
    \label{equ:artVeh_approx}
    \dot{\bm{z}} = \bar{\bm{f}}(\bm{z},\bm{u})=
    \begin{bmatrix} 
        1 \\   
        0 \\
        - z_2 \\
        - z_3
    \end{bmatrix}
    v_{1} + 
    \begin{bmatrix} 
        0 \\   
        1 \\
        0 \\
        - \frac{1}{4} z_2^2
    \end{bmatrix}
    v_{2}.
    \end{equation}
\end{example}
\begin{example}
    The homogeneity condition~\eqref{eq:homogeneity_condition_definition} for the  nilpotent approximation~\eqref{equ:artVeh_approx} of the articulated vehicle is given by
    \begin{align}
    \label{eq:artVeh{eq:artVeh_homogeneousCond}_homogeneousCond}  \bar{\bm{f}}(\bm{\Lambda}_\alpha\bm{z},&\bm{\Delta}_{\alpha}\bm{v})
    = \begin{bmatrix}
        \alpha^{s_1}v_1 \\
        \alpha^{s_2}v_2 \\
        - \alpha^{r_2+s_1}z_2 v_1 \\
        - \alpha^{r_3+s_1}z_3 v_1 - \frac{1}{4} \alpha^{2 r_2+s_2}z_2^2 v_2 \\
    \end{bmatrix}\nonumber
    \\
    {}={}&\alpha^{\tau}\bm{\Lambda}_{\alpha} \bar{\bm{f}}(\bm{z},\bm{v})
    = \alpha^{\tau}
    \begin{bmatrix}
        \alpha^{r_1}v_1 \\
        \alpha^{r_2}v_2 \\
        - \alpha^{r_3}z_2 v_1 \\
        - \alpha^{r_4} \left( z_3 v_1 + \frac{1}{4} z_2^2 v_2 \right) \\
    \end{bmatrix}.
    \end{align}
    For the choice of parameters $\bm{r} = \w =[1 \ 1 \ 2 \ 3]$, $\bm{s} = \bm{1}_{2}$, and $\tau=0$, this condition is satisfied for all $(\bm{z},\bm{v})\in\R^{4}\times\R^{2}$ and all $\alpha\in\R_{\geq 0}$. Hence, the approximation is $(\bm{w},\bm{1}_2,0)$-homogeneous.
\end{example}
%
%
%
%
%
\onecolumn{\bibliography{literature}}
\bibliographystyle{ieeetr}

\end{document}